\documentclass[reqno,a4paper, 12pt]{amsart}

\usepackage{bm}
\usepackage{amsmath}
\usepackage{amsfonts}
\usepackage{amssymb}
\usepackage{amsthm}
\usepackage{graphicx}
\usepackage{color}
\usepackage{enumitem}
\usepackage{url}
\usepackage{a4wide}
\theoremstyle{plain}

\newtheorem{theorem}{Theorem}[section]
\newtheorem{corollary}[theorem]{Corollary}
\newtheorem{definition}[theorem]{Definition}
\newtheorem{example}[theorem]{Example}
\newtheorem{lemma}[theorem]{Lemma}

\newtheorem{proposition}[theorem]{Proposition}
\newtheorem{remark}[theorem]{Remark}

\newtheorem{conjecture}[theorem]{Conjecture}

\numberwithin{equation}{section}
\newcommand{\tr}{\mathop{\mathrm{tr}}}

\newcommand{\diag}{\mathop{\mathrm{diag}}}
\newcommand{\rank}{\mathop{\mathrm{rank}}}

\newcommand{\signum}{\mathop{\mathrm{sgn}}}

\usepackage{hyperref}

\numberwithin{equation}{section}
\begin{document}

\begin {center}
      CIMPA WORKSHOP\\
   Analytical and algebraic tools of modern multivariate analysis and graphical models \\
   {\it Outils analytiques et algebraiques modernes\\
   de la statistique multivariate et des modules graphiques}\\
 2011, Hammamet,  Tunisia\\[40mm]

{\LARGE \bf STOCHASTIC
 ANALYSIS METHODS\\ IN WISHART THEORY\\[10mm]}

  {\Large
  by\\[1mm]
  Piotr Graczyk}\\[1mm]

  {\large LAREMA, Universit\'e d'Angers, France}\\[2mm]

  { \Large and\\[1mm]

   Eberhard Mayerhofer}\\[1mm]

 {\large Deutsche Bundesbank, Frankfurt am Main, Germany}\\[20mm]
  \end{center}
 %%%%%%%%%%%%%%%%%%%%%%%%%%%%%%%%%%%%%%%%%%%%%%%%%%%%%%%%%%%%%%%%%%%%%%%%%%%%%%%%%%%%%%
  {\large {\bf  Part I.   Yamada-Watanabe Theorem for Matrix Stochastic Differential Equations.
  Moments of Wishart Processes. {\it By P. Graczyk} }\\[2mm]
 %%%%%%%%%%%%%%%%%%%%%%%%%%%%%%%%%%%%%%%%%%%%%%%%%%%%%%%%%%%%%%%%%%%%%%%%%%%%%%%%%%%%%%%%
    {\large {\bf  Part II.   Wishart Processes and  Wishart Distributions: an Affine Processes Point of View. }  {\it By E. Mayerhofer} }}\\[2mm]

\vfill\eject
%%%%%%%%%%%%%%%%%%%%%%%%%%%%%%%%%%%%%%%%%%%%%%%%%%%%%%%%%%%%%%%%%%%%%%%%%%%%%%%%%%%%%%%%%%%%%%%
\begin{center}
 {\large {\bf  Part II\\[2mm]
 WISHART PROCESSES AND WISHART DISTRIBUTIONS: AN AFFINE PROCESSES POINT OF VIEW\\[2mm]
  {\it by E. Mayerhofer} }}\\[10mm]
  \end{center}
\tableofcontents
\section{Introduction}\label{sec: intro}In his 1928 Biometrika contribution \cite{Wishart}, Wishart introduced the distribution of covariance matrices of samples from a normally distributed random variable. His contribution triggered a lot of research on theory and application of these and other related multivariate distributions in e.g., multivariate statistics, probability theory and most recently in finance and financial mathematics (for an account of the
literature, see \cite{CFMT} and the references therein). An important subclass of Wishart distributions arise as pushforwards of normal distributions under certain quadratic forms: Let $\xi_1,\xi_2,\dots,\xi_k$ be a sequence of $\mathbb R^d$--valued independent, normally distributed random variables
with mean vectors $\mu_i\in\mathbb R^d$ and covariance matrix $\Sigma$. Then
\[
\Xi:=\xi_1\xi_1^\top+\dots+\xi_k \xi_k^\top
\]
is Wishart distributed  with scale parameter $p:=k/2$, shape parameter $\sigma:=2\Sigma$ and parameter of non-centrality $\omega:=\sum_{i=1}^k \mu_i \mu_i^\top$. We use the
notation $\Gamma(p,\omega;\sigma)$ from \cite{MayerhoferExistenceNonCentral} for the distribution of $\Xi$ which in turn is motivated by Letac and Massam's \cite{Letac08} family of
Wishart distributions. Note that $\Xi$ is positive semidefinite almost surely, and also
$\omega$ and $\sigma$ are positive semidefinite matrices. Considering first the one-dimensional case $d=1$ and assuming the non-trivial case $\Sigma\neq 0$, we can calculate the Laplace transform (or quite similiarly the characterteristic function) of $\xi_j\xi_j^\top=\xi_j^2$, by mere completion of a square,
\begin{align*}
\mathbb E[e^{-u\xi_j^2}]&=\frac{1}{\sqrt{2\pi \Sigma}}\int_{\mathbb R}e^{-u\eta^2-\frac{(\eta-\mu_j)^2}{2\Sigma}} d\eta =\\
&=\frac{1}{\sqrt{2\pi \Sigma}}\int_{\mathbb R}\exp\left(-\frac{1+2\Sigma u}{2\Sigma}\left(\eta-\frac{\mu_j}{1+2\Sigma u}\right)^2
-\frac{1}{2\Sigma}\mu_j^2\left(1-\frac{1}{1+2\Sigma u}\right)\right)d\eta\\
&=\frac{1}{\sqrt{1+2\pi\Sigma}}\underbrace{\frac{1}{\sqrt{2\pi \frac{\Sigma}{1+2\Sigma u}}}\int_{\mathbb R}\exp\left(-\frac{1+2\Sigma u}{2\Sigma}\left(\eta-\frac{\mu_j}{1+2\Sigma u}\right)^2\right)d\eta}_{\textrm{ integral over density of a normal distribution} =1 } \times e^{-u\mu_j^2 /(1+2\Sigma u)}\\
&=\frac{e^{-u (1+2\Sigma u)^{-1} \mu_j^2}}{(1+2\Sigma u)^{1/2}},\quad u\geq 0.
\end{align*}
Here $\mathbb E[\cdot]$ denotes the expectation operator on the respective probability space which supports the random variables $\xi_j$. By the independence of $\xi_j$, ($j=1,\dots,k$), we obtain
\[
\mathbb E[e^{-u\Xi}]=\prod_{j=1} ^k \mathbb E[e^{-u\xi_j^2}]=\frac{e^{-u (1+\sigma u)^{-1}\omega}}{(1+\sigma u)^{p}},\quad u\geq 0.
\]
We see that the family of distributions $\Gamma(p,\omega;\sigma)$  is a natural extension of non-central chi-square distributions on the one hand, where
 $\Sigma=1$,  and of gamma distributions on the other hand, where $\omega=0$.
Let now $d>1$. In the following we denote by $\overline{S}_d^+$
the cone of positive semidefinite matrices, by $\tr(A)$ the trace of a $d\times d$ matrix $A$, and by $\det(A)$ the matrix determinant. The positive definite matrices are abbreviated as $S_d^+$.

With this notation,
a similar calculation as above (taking into account the non-\\ commutativity of the matrix multiplication) yields the formula for the Laplace transform
\begin{equation}\label{eq: discrete wishart dist}
\mathbb E[e^{-\tr(u\Xi)}]=\frac{e^{-\tr(u (I+\sigma u)^{-1}\omega)}}{\det(I+\sigma u)^{p}},\quad u\in \overline S_d^+.
\end{equation}
Here $I$ denotes the unit $d\times d$ matrix, $ab$ denotes the matrix product of matrices $a,b$ and $a^{-1}$ is the inverse of a non-degenerate matrix $a$.

It is well known that chi-square distributions and gamma distributions exist for all shape parameters $p\geq 0$. It therefore may be conjectured that the same holds true in dimensions $d\geq 2$. However, this is
not the case. A number of authors from different scientific communities (see the references given in \cite{PeddadaRichardsPriority}) proved that for invertible $\sigma$, the central Wishart distributions $\Gamma(p;\sigma):=\Gamma (p,\omega=0;\sigma)$ exist if and only if $p$ belongs to the Gindikin ensemble, which equals the set
\begin{equation*}
\Lambda_d:=\{0, 1/2,\dots, (d-1)/2\}\cup (\frac{d-1}{2},\infty).
\end{equation*}
In other words, the right side of eq.~\eqref{eq: discrete wishart dist} is the Laplace transform of a distribution on $\overline S_d^+$ if and only if  $p\in\Lambda_d$.

Note that for $d\geq 3$ the set consists of a discrete part and a continuous part. Of course, $\Gamma(p=0;\sigma)$ is trivial (the unit mass at the origin). Some authors therefore exclude $0$ from the Gindikin set. The non-central case $\omega\neq 0$ is more complicated, but also more interesting.
Peddada and Richards \cite{PeddadaRichards91} show by using technically complicated but elementary arguments involving special functions that if $\rank(\omega)=1$, then $p\in\Lambda_d$. The general  case (arbitrary rank) has been understood completely very recently. While \cite[Proposition 2.3]{Letac08} conjectures the same characterization holds for non-central Wishart distributions, the author of the present note has shown in \cite{MayerhoferExistenceNonCentral}  that for $p<\frac{d-1}{2}$, also the parameter of non-centrality must be of lower rank, namely $\rank(\omega)\leq 2p+1$ (see Theorem \ref{th: maintheorem} below). Furthermore, a preliminary version of that paper, \cite{MayerhoferExistenceOld}, conjectured that in this case $\rank(\omega)\leq 2p$ (subsequently it turned out that
the method I use only implies the weaker rank condition $\rank(\omega)\leq 2p+1$). Very recently, Letac and Massam \cite{Letac11}  confirm my conjecture, while they falsify theirs (see Theorem \ref{th: letac11}).

There is a dynamic way to generate noncentral chi-square distributions. Namely, by taking a $k$--dimensional standard Brownian motion $(B^1_t,\dots, B^k_t)^\top$
and some initial value $y=(y_i)_{i=1}^k\in\mathbb R^k$, we see that the non-negative stochastic process $X_t:=(y+\sqrt{\Sigma} B_t)^\top (y+\sqrt{\Sigma} B_t)$ is distributed according to
$\Gamma(k/2, x; 2\Sigma t)$, with initial value $X_0=x=y^\top y\geq 0$. This follows from the fact that $y_i+\sqrt{\Sigma} B^i_t$ are independent, normally distributed random variables
with mean $y_i$ and variance $\Sigma t$. Processes constructed this way are termed ``
Square Bessel Processes'', and it can be shown that they are well defined
also for any non-negative parameter $p\geq 0$. For $\Sigma=1$ and $\delta=2p$, Pitman and Yor \cite{PitmanYor} denote this class as $W_{\delta}^x$.

The matrix-variate generalization of these Square Bessel Processes are the so-called Wishart processes introduced by Bru \cite{bru}. Their crucial feature is the affine property: Their Laplace transform is exponentially affine in the initial state, $X_0=x$. A modern way of looking at Wishart processes is by considering them as a subclass of affine processes
on positive semi-definite matrices or subsets thereof, while the traditional way originating from Bru and followed by others is of solving certain stochastic differential equations (in these notes they will be called Wishart SDEs). These lecture notes try to explain the connections between the two viewpoints. See also section \ref{subsec: wish semimartin}.

We also discuss the existence of Lebesgue densities for Wishart distributions as well as the existence of transition densities for Wishart processes. Final remarks are on the existence of Wishart processes on state-spaces different from the positive semi-definite matrices.

\section{Wishart semimartingales, Wishart distributions and Wishart processes}\label{sec: relations}
In this section we introduce and comment on the three main objects of this article: {\bf Wishart semimartingale}, their marginal distributions, which are {\bf Wishart distributions}, and the{ \bf Wishart processes}, which in these notes
are Markov processes having so-called {\bf Wishart transition laws}.  We will show that Wishart semimartingales can be realized as solutions to {\bf Wishart SDEs}, and that Wishart processes are actually Wishart semimartingales.
\subsection*{Wishart semimartingales}\label{subsec: wish semimartin}
\begin{definition}
Let $(\Omega,\mathcal F,\mathcal F_t,\mathbb P)$ be a filtered probability space satisfying the usual conditions. Let $p\geq 0$, $\beta$  be a $d\times d$ matrix and let
further $\alpha\in \overline S_d^+$. A continuous semimartingale $X_t$
is called Wishart semimartingale with parameters $(\alpha, p, \beta)$, if we can write $X_t=x+D_t+M_t$, where $X_0=x$, $D_t=\int_0^t (2p \alpha+\beta X_s+X_s\beta^\top)ds$
and $M_t$ is a local martingale with quadratic variation
\begin{equation}\label{Wish q var}
[M_{t,ij}, M_{t,kl}]=\int_0^t\left((X_s)_{ik}\alpha_{jl}+(X_s)_{il}\alpha_{jk}+(X_s)_{jk}\alpha_{il}+(X_s)_{jl}\alpha_{ik}\right)ds,
\end{equation}
\end{definition}
It follows immediately that $M_t$ is continuous with $M_0=0$ a.s., and $D_0=0$ a.s., as well. An important class of Wishart semimartingales are those obtained by certain squares of
Ornstein-Uhlenbeck processes. These correspond to drift parameters $2p\in\mathbb N$:
\begin{example}\label{wish 1 const}
Let $p\in \mathbb N/2$, $m:=2p$, $\alpha\in \overline S_d^+$ and $\beta$ a $d\times d$ matrix.  Choose a $d\times d$ matrix $Q$ for which $Q^\top Q=\alpha$ (there are, in general,
arbitrary many choices for $Q$). For $i=1,\dots,m$ we define
\[
Y_{i,t}:=e^{\beta t}\left(y_i+\int_0^te^{-\beta s}Q^\top dW_s^i\right), \quad t\geq 0,
\]
where $W=(W^1,\dots, W^m)$ is a $d\times m$ standard Brownian motion, and $W^i$ is an $d$-dimensional column vector of standard Brownian motions, and $y_i\in \mathbb R^d$.

Then $Y_i$ is a Gaussian process for every $i=1,\dots,m$. In fact, $Y_{i}$ is an OU-process solving the stochastic
differential equation
\[
dY_{i,t}=\beta Y_{i,t}+Q^\top dW_t^i,\quad Y_{i,0}=y_i.
\]
We define the continuous semimartingale $X_t:=\sum_{i=1}^m Y_{i,t} Y_{i,t}^\top$. Then $X_t$ starts at $X_0=\sum_{i=1}^m y_i y_i^\top$, and we have that
\[
dX_t=(2p Q^\top Q+\beta X_t+X_t\beta^\top)dt+dM_t,
\]
where $M_t$ consists of Brownian terms only. A straightforward calculation yields that $M_t$ has quadratic variation \eqref{Wish q var}, where we have to define $\alpha=Q^\top Q$. Hence
$X_t$ is a Wishart semimartingale with parameters $(\alpha, p, \beta)$.
\end{example}

\begin{example}
The following is a particular case of the preceding example ($Q=I, \,\beta=0$), but written in matrix form.
Let $W$ be a $d\times n$ matrix valued Brownian motion, where $n\geq d$. That is, the entries of $W$ consist of $d\times n$ independent standard Brownian motions. Let further $x=yy^\top$. Then, as can be calculated using It\^o-calculus, the process
$X_t:=(y+W)(y+W)^\top$ is an $\overline S_d^+$-valued Wishart semimartingale with $2p=n$,
$\alpha=I$, $\beta=0$, starting at $X_0=x$.
\end{example}
The following note aims at those readers, who are already accustomed to Wishart processes in the sense of Bru:
\begin{remark}\rm
It is not so trivial to write $X_t$ as solution of a Wishart SDEs, which are defined below in equation eq.~\eqref{wishart full class}. The main technical problem is
to derive from $W$ and $Y$ a new Brownian motion $B$, for which $X$ satisfies  the stochastic differential equation \eqref{wishart full class}. For $2p\geq d+1$,
Pfaffel \cite[Theorem 4.19]{pfaffelProjekt} succeeds by using L\'evy's characterization of Brownian motion. For $2p<d+1$ one can show this by an appropriate enlargement of the underlying probability space.
See statement and proof of Lemma \ref{wish sde is wishart semi}. This technical problem supports our decision to introduce the notion of Wishart semimartingale through these notes.
A further and independent motivation is coming from the recent affine processes literature, where the notion of affine semimartingale appears \cite{Kallsen}. In our case, the affine character
of Wishart semimartingales is reflected by the instantaneous drift $dD_t/dt$ and the instantaneous quadratic variation $dM_t/d_t$, which are both affine function in the state $X_t$.
The second and important affine character of Wishart semimartingales is constituted by their exponentially affine Laplace transform, see Lemma \ref{lem wish semi is wish dist}
below.
\end{remark}

\subsection*{Wishart semimartingales are solutions to Wishart SDEs}
We now relate this class of semimartingales to solutions of certain stochastic differential equations (SDEs).

Let $\sqrt A$ denote the unique square root of $A\in\overline S_d^+$. Let $Q,\beta$ be real valued $d\times d$ matrices and
$p\geq 0$. As {\bf Wishart SDE} we define the stochastic differential equation
 \begin{equation}\label{wishart full class}
dX_t=\sqrt {X_t} dB_t Q+Q^\top dB_t^\top \sqrt {X_t}+(2p\,Q^\top Q+\beta X_t+X_t\beta^\top )dt,\quad X_0=x\in \overline S_d^+,
\end{equation}
where $B$ is a standard $d\times d$ Brownian motion.

We understand any solution of \eqref{wishart full class} as {\it weak solution}, which means that for given $Q$, $\beta$ and $p$, there exists a filtered probability space $(\Omega,\mathcal F,\mathcal F_t,\mathbb P)$ which supports
a pair of $\mathcal F_t$ adapted processes $(X, B)$, which satisfy the It\^o-integral equation \eqref{wishart full class} (which, as is common, is written in differentials). However, if the probability space as well as $B$
are given in advance, then any solution of \eqref{wishart full class} is  called a {\it strong} one. While every strong solution yields a weak solution, the converse does not hold in general. An example of a stochastic differential equation (SDE) which admits weak\footnote{Let $X=B$ be a standard Brownian motion, then  $dW_t=\signum(B_t) dB_t$ is a new Brownian motion in virtue of L\'evy's characterization of Brownian motion, and we have $dX_t=\signum(X_t)dW_t$ and clearly $X_0=B_0=0$ because $B$ is standard.} but not strong solution is Tanaka's one-dimensional equation
\cite[Example 5.3.2]{Oksendal}
\[
dY_t=\signum(Y_t)dW_t,\quad Y_0=0.
\]

While the symmetrization in equation \eqref{wishart full class} is necessary to guarantee a stochastic evolution
on the space of symmetric $d\times d$ matrices, the existence of solutions for \eqref{wishart full class} is far from trivial.
In fact, it is not quite straightforward to ensure that for positive semidefinite starting values $X_0=x\in \overline S_d^+$,
a local solution exists: If we assume $x\in S_d^+$, then a solution exists at least for an almost surely strictly positive stopping time $T_x$. This is the first hitting time of the boundary, for $X$. The solution is a strong one, and its existence follows  from the fact that the matrix square root is locally Lipschitz
on $S_d^+$. However, in general, the interval $[0, T_x]$ is purely stochastic, i.e., we might have $\inf_{\omega\in \Omega} \{T_x(\omega)\mid T_x(\omega)>0\}=0$.  This definitely happens when $p<\frac{d-1}{2}$ and under the premise that $Q$ is non-degenerate, see Lemma \ref{lem definite hit} and Lemma \ref{stronger lem definite hit}. For $p\geq \frac{d+1}{2}$, it has been shown in \cite{pfaffel} that $T_x=\infty$ almost surely. In our terminology this means in particular that for each $(\alpha,\beta, p)$ with $p\geq \frac{d-1}{2}$ a Wishart semimartingale exists. When $p\in [\frac{d-1}{2},\frac{d+1}{2})$ and in special cases, particularly when $\beta=0$ and $Q$ is invertible, Graczyk and Malecki \cite{GraczykMalecki} provide the existence of strong solutions with different methods than this note (see also the first part of these lecture notes written
by P. Graczyk).

We allow for starting values $x\in \overline S_d^+$, hence in particular we allow that the process $X_t$ starts at the boundary $\partial \overline S_d^+$ of $\overline S_d^+$
which are precisely the positive semidefinite matrices with rank strictly smaller than $d$. As the square root is not Lipschitz on $\partial \overline S_d^+$,
standard existence results do not apply. We will see however in a moment, how to infer weak solutions from the existence of Wishart semimartingales:

\begin{lemma}\label{wish sde is wishart semi}
Any solution of the Wishart SDE \eqref{wishart full class} is a Wishart semimartingale. Conversely, suppose $X_t$ is a Wishart semimartingale. Then there exists an enlargement of
$(\Omega, \mathcal F,\mathcal F_t,\mathbb P)$ which supports a $d\times d$ Brownian motion and some $d\times d$ matrix $Q$ for which $Q^\top Q=\alpha$ such that $X_t$ is a solution of the Wishart SDE \eqref{wishart full class}.
\end{lemma}
\begin{proof}
Clearly, the solution of the Wishart SDE is an It\^o-process with instantaneous drift $b(X_t)=(2p\,Q^\top Q+\beta X_t+X_t\beta^\top )$. By definition, $X_t$ is the sum of a local martingale (the Brownian terms) and a process of finite variation (the integrated drift) plus initial value, $X_t=x+D_t+M_t$.
Furthermore, writing out the Brownian terms of \eqref{wishart full class} in coordinates (and using Einstein's summation convention, where summation is performed over all indices which appear twice), we have
\[
dM_{t,ij}=(\sqrt {X_t})_{ir}dB_{t,rs}Q_{sj}+Q_{ri}d B_{t,sr}(\sqrt{X_t})_{sj}.
\]
Hence, using the formal rules $d[B_{t,ab},B_{t,cd}]=0$ if $(a,b)\neq (c,d)$ and $d[B_{t,ab},B_{t,cd}]=dt$
if $(a,b)=(c,d)$, we have
\[
d[X_{t,ij} X_{t,kl}]=d[M_{t,ij}, M_{t,kl}]=\left((X_t)_{ik}\alpha_{jl}+(X_t)_{il}\alpha_{jk}+(X_t)_{jk}\alpha_{il}+(X_t)_{jl}\alpha_{ik}\right)dt,
\]
where we have set $\alpha=Q^\top Q$.  Hence we see that $X_t$ is a Wishart semimartingale.

The converse direction is proved in full generality in \cite{CFMT}. To avoid technicalities (which only arise in view of the multivariate character of the problem), and
to see the essence of the problem, we just consider the case $d=1$ here. This is also in some way a prelude foreplay for what is demonstrated in more generality
in section \ref{MCKean argument}.

We have $X_t=X_0+D_t+M_t$, where $X_0=x$, $d D_t=(b+\beta X_t)dt$, $b\geq 0$ and $d[M,M]_t=\sigma^2 X_t dt$.

Suppose first $x>0$, and $b\geq \sigma^2/2$. Using It\^o-calculus we see that $Y_t=\log(e^{-\beta t}X_t)$ satisfies
\[
dY_t=-\beta dt+\frac{1}{X_t}(dX_t+dM_t)-\frac{1}{2X_t^2}\sigma^2 X_tdt=\frac{b-\sigma^2/2}{X_t}dt+dM_t/X_t,
\]
which equals the differentials of a non-negative process plus a continuous local martingale. If $X_t$ would hit zero in finite time,
then $Y_t$ would go to $-\infty$ in finite time. Because the first summand above is non-negative, this carries over to the second one. But $\int_0^t X_s^{-1}dM_s$
 is actually just a time changed Brownian motion,
hence oscillates infinitely often (and a.s.) between $-\infty$ and $+\infty$. It can not go to $-\infty$ in finite time! So we see that $X_t$ is strictly positive a.s., and for all $t\geq 0$. Now we can invert $X_t$, and therefore the process $B_t$ defined by
\[
dB_t:=\frac{dX_t-dD_t}{\sigma\sqrt{X_s}}=\frac{dX_t-(b+\beta X_t)dt}{\sigma\sqrt{X_t}}=\frac{dM_t}{\sigma\sqrt{X_t}}, \quad B_0:=0,
\]
is a well defined continuous local martingale and by construction, $[B_t,B_t]=t$ a.s., for all $t\geq 0$. L\'evy's continuity theorem applies and yields that
$B_t$ is a standard Brownian motion. Rewriting the definition of $B_t$ yields that $X_t$
is a solution of the Wishart SDE
\[
dX_t=(b+\beta X_t)dt+\sigma \sqrt{X_t}dB_t,\quad X_0=x.
\]

In the general case (where $X_t$ may hit zero in finite time, or even start there),
one must in general enlarge the probability space to obtain $X_t$ as solution of
a corresponding Wishart SDE. To this end we use the arguments of \cite[Theorem V.20.1]{rogerswilliams}, which are much simpler
in the case $d=1$. Let $(\Omega,\mathcal G,\mathcal G_t,\mathbb P)$ be an enlargement of the current probability space which supports a standard Brownian motion  $W$
independent of $X$. We define the process
\[
\widetilde B_t:=\int_0^t \theta_s dM_s+\int_0^t\rho_s dW_s,
\]
where $\theta$ and $\rho$ are the predictable processes
\[
\theta_t:=\frac{1}{\sigma \sqrt{X_s}}1_{X_s>0},\quad \rho_t:=1_{X_t=0}.
\]
Then by construction $[\widetilde B,\widetilde B]_t=t$ and $\widetilde B$ is a continuous local martingale starting at $0$.
Hence, by L\'evy's characterization, it is standard Brownian motion. Furthermore
\[
dX_t=dM_t+dD_t=\sigma \sqrt{X_t} d\widetilde B_t+(b+\beta X_t)dt,
\]
which is just the Wishart SDE in the one-dimensional situation. In the multivariate case, $\theta$ and $\rho$ are vectors, whose construction
is due to \cite[Lemma V.20.7]{rogerswilliams}.
\end{proof}

One can show with very little effort that $M$ is an $L^2$ martingale, for instance by using the fact
that $X_t$ is Wishart distributed (see Lemma \ref{lem wish semi is wish dist}), since the Wishart distribution exhibits exponential moments.
\subsection*{Wishart semimartingales are Wishart distributed}
First, we define the family of Wishart distributions, which is motivated by the derivation of \eqref{eq: discrete wishart dist}:
\begin{definition}
We define the non-central Wishart distribution
$\Gamma(p,\omega;\sigma)$ on the space of symmetric $d\times d$ matrices $S_d$ --whenever it exists--by its Laplace transform
\begin{equation}\label{FLT Mayerhofer Wishart}
\mathcal L (\Gamma(p,\omega;\sigma))(u)= \left(\det(I+\sigma u)\right)^{-p}e^{-\tr(u(I+\sigma
u)^{-1}\omega)},\quad u\in \overline S_d^+,
\end{equation}
where $p\geq 0$ denotes its shape parameter, $\sigma\in \overline S_d^+$ is the
scale parameter and the parameter of non-centrality equals
$\omega\in \overline S_d^+$.
\end{definition}
\begin{lemma}\label{lem: support wishart}
Any Wishart distribution $\Gamma(p,\omega;\sigma)$ is supported on $\overline S_d^+$.
\end{lemma}
\begin{proof}
It suffices to show that for any $v\in \mathbb R^d$, we have that
the push forward $\Pi_*$ of $\Gamma(p,\omega;\sigma)$ under the map
$ \Pi: S_d\rightarrow \mathbb R$, $x\mapsto v^\top x v$ is supported on $\mathbb R_+$. This, in turn, follows from the fact that $\Pi_*(\Gamma(p,\omega;\sigma)(d\xi)$ is non-centrally gamma distributed: By Proposition \ref{proplemma} \ref{firstelemprop}, we may assume $\sigma=2 I$ without loss of generality. In the following we use $\lambda$ as the Laplace variable, and we let $U$ be an orthogonal matrix and $\mu\geq 0$
such that $v=\mu U e_1$, where $e_1$ is the first canonical basis vector of $\mathbb R^d$. Accordingly,
$\omega'=U^\top \omega U$. The Laplace transform of $\Pi_*(\Gamma(p,\omega;\sigma))(d\eta)$ equals
\begin{align*}
&\det(I+2 \lambda v v^\top)^{-p} e^{-\tr(\lambda v v^\top (I+2\lambda v v^\top)^{-1}\omega)}\\
&\quad=\det(I+2\lambda\mu^2 e_1e_1^\top)^{-p}e^{-\tr (\mu^2e_1e_1^\top(1+2\lambda\mu^2 e_1e_1^\top)^{-1}\omega)) }=(1+2\lambda\mu^2)^{-2p/p}e^{-\lambda \mu^2(1+2\lambda \mu^2)^{-1} \omega'_{11}},
\end{align*}
which is the Laplace transform of $\mu ^2 X$, where $X$ is a  non-central chi-square distributed random variable with shape parameter $2p$ and parameter of non-centrality $w'_{11}$.
\end{proof}
Suppose $\beta$ is again a $d\times d$ matrix with real entries, and let $\alpha\in \overline S_d^+$.
In the following we denote by $\omega^{\beta}_t$ the flow of the vector field $\beta x+x\beta^\top$, that is,
\[
\omega^{\beta}:\,\, \mathbb R\times \overline S_d^+\rightarrow \overline S_d^+,\quad\omega^{\beta}_t(x):=e^{\beta t}x e^{\beta^\top t}.
\]
Its twofold integral is denoted by
\[
\sigma^{\beta}:\,\, \mathbb R_+\times \overline S_d^+\rightarrow \overline S_d^+,\quad\sigma^{\beta}_t(y)=2\int_0^t \omega^{\beta}_s(y) ds.
\]
Using these two functions, we define a curve $\phi(t,\cdot)$ and a matrix-valued curve $\psi(t,\cdot)$

\begin{align}\label{formula: Wishartphi}
\phi(t,u)&=p\log\det \left(I+u
\sigma_t^\beta(\alpha)\right),\\\label{formula: Wishartpsi}
\psi(t,u)&=e^{\beta^\top
t}\left(u^{-1}+\sigma_t^\beta(\alpha)\right)^{-1}e^{\beta t}.
\end{align}

We show now the elementary fact:
\begin{proposition}\label{prop: ric eq}
 $\phi$ and $\psi$ satisfy a system of generalized Riccati equations,
namely,
\begin{align}\label{eq: phi}
\dot{\phi}(t,u)&=2p\,\tr( \alpha \psi(t,u)),\quad\phi(0,u)=0,\\\label{eq: psi}
\dot{\psi}(t,u)&=-2u\,\alpha u+\psi(t,u)\beta+\beta^\top\psi(t,u),\quad\psi(0,u)=u.
\end{align}
\end{proposition}
\begin{proof}
In order to obtain the generalized Riccati equations
\eqref{eq: phi}--\eqref{eq: psi}, we differentiate the formula \eqref{formula: Wishartpsi} for $\psi$ by using the fact that
for any differentiable matrix-valued curve $t\mapsto a(t)$ we have
$\frac{d}{dt}a^{-1}(t)=-a^{-1}(t) \frac{d}{dt}a(t)a^{-1}(t)$, see for instance
\cite[Proposition III.4.2 (ii)]{Faraut}. Formula \eqref{formula: Wishartphi} is obtained by using the rule $\frac{d}{dt}\log(\det(a(t))=\tr(a^{-1}(t)\frac{d}{dt}a(t))$, see  \cite[Proposition II.3.3 (i)]{Faraut}.
\end{proof}

The following is proved in \cite{bru}, but with different notation, and for solutions to Wishart SDEs.
The statement, however, is in fact a result concerning the law of Wishart semimartingales:
\begin{lemma}\label{lem wish semi is wish dist}
Suppose $X_t$ is a Wishart semimartingale with parameters $(\alpha, p,\beta)$ starting at $X_0=x$. Then for each $t\geq 0$, $X_t\sim \Gamma(p, \omega^\beta_t(x);\sigma^\beta_t(\alpha))$.
\end{lemma}
\begin{proof}
Let $t>0$ and $u\in \overline S_d^+$. Let $(\phi,\psi)$  be the functions defined by eqs.~\eqref{formula: Wishartphi}--\eqref{formula: Wishartpsi}. Applying the It\^o-formula to the process
\[
J_s:=e^{-\phi(t-s,u)-\tr(\psi(t-s,u) X_s)},\quad 0\leq s\leq t,
\]
and using thereby Proposition \ref{prop: ric eq}, we obtain
\begin{align*}
\frac{dJ_s}{J_s}&=(\partial_s\phi(t-s,u)+\tr(\partial_s\psi(t-s,u)X_s))ds-\tr(\psi(t-s,u)((\beta X_s+X_s\beta^\top+2p Q^\top Q)ds+dM_s))\\
&+\frac{1}{2}4\tr(\psi(t-s,u)\alpha\psi(t-s,u)X_s)\\
&=(\partial_s\phi(t-s,u)-2p \tr(Q^\top Q\psi(t-s,u)))ds-\tr(\psi(t-s,u)dM_s)\\
&+\tr(X_s (\partial_s\psi(t-s,u) +2\psi(t-s,u)\alpha\psi(t-s,u)-\psi(t-s,u)\beta+\beta^\top\psi(t-s,u)))\\
&=-\tr(\psi(t-s,u)dM_s),
\end{align*}
where the first two brackets vanish because of equations \eqref{eq: phi}--\eqref{eq: psi}. We conclude that
$(J_s)_s$ is a local martingale on $[0,t]$. Furthermore, since $\phi(t,u)\geq 0$ for all $t\geq 0$
and $\psi(t,u)\in \overline S_d^+$ for all $t\geq 0$, we have that $J$ is uniformly bounded on $[0,t]$. Hence
$J$ is a true martingale, and therefore
\[
\mathbb E[e^{-u X_t}\mid X_0=x]=\mathbb E[J_t\mid X_0=x]=J(0)=e^{-\phi(t,u)-\tr(\psi(t,u)x)},
\]
where we have used that $J_t=e^{-\tr(uX_t)}$ (which follows from $\phi(0,u)=0$ and $\psi(0,u)=u$).
The assertion concerning the distribution of $X_t$ now follows from the explicit formulas
\eqref{formula: Wishartphi}--\eqref{formula: Wishartpsi} and the very definition of
the Wishart distribution \eqref{FLT Mayerhofer Wishart}.

For the derivation of the exponentially affine characteristic function on general state-spaces, see the proof of \cite[Theorem 2.2]{ADPTA}, which uses similar arguments.
\end{proof}
\subsection*{Wishart processes from the Markovian viewpoint}

\begin{definition}\label{def wish pro m}
A family of distributions $\{p_t(x,d\xi)\mid t\geq 0,x\in \overline S_d^+\}$  which is non-centrally Wishart distributed according to
\begin{equation}\label{def: Wishart}
p_t(x,d\xi)=\Gamma(p,\omega_t^\beta(x);\sigma_t^\beta(\alpha))(d\xi)
\end{equation}
is termed Wishart transition function with constant drift parameter $p\geq 0$, linear drift parameter $\beta$ and
diffusion coefficient $\alpha\in \overline S_d^+$.
\end{definition}

By using the Laplace transform of the Wishart distribution, we obtain that
the Laplace transform of the laws $p_t(x,d\xi)$ is given by
\begin{align}\label{eq: FLT Bru}
\int_{\overline S_d^+}e^{-\tr(u\xi)}p_t(x,d\xi)&=\left(\det(I+\sigma^{\beta}_t(\alpha)u) \right)^{-p}
e^{-\tr\left(u\left(I+\sigma^{\beta}_t(\alpha)
u\right)^{-1}\,\omega^{\beta}_t(x)\right)}\\\label{affine def mark}
&=e^{-\phi(t,u)-\tr( \psi(t,u) x)}, \quad u\in \overline S_d^+,
\end{align}
where $(\phi,\psi)$ are of the same form as in \eqref{formula: Wishartphi}--\eqref{formula: Wishartpsi}.

We start with the following observation.
\begin{lemma}\label{wish is mark}
Any Wishart transition function is a Markovian transition function supported on $\overline S_d^+$. The associated Markovian semigroup $(P_t)_{t\geq 0}$ defined by
\begin{equation}\label{markov semi}
f\mapsto P_t f(x):=\int_{\overline S_d^+} f(\xi) p_t(x,d\xi)
\end{equation}
is a Feller semigroup, that
is, $P_t$ reduces to a strongly continuous contraction semigroup acting on $C_0(\overline S_d^+)$, the continuous functions on $\overline S_d^+$ vanishing at infinity.
\end{lemma}
We use the terminology {\bf Wishart process} for Markov processes with Wishart transition function.
\begin{proof}
First, by Lemma \ref{lem: support wishart} we know that for all $t\geq 0$, $x\in \overline S_d^+$, the laws $p_t(x,d\xi)$ are supported on $\overline S_d^+$. For any Borel set $B$, measurability of $p_t(x,B)$ in $(t,x)$
holds by construction (and in view of the continuity of the maps $\sigma_t^\beta(\alpha),\, \omega_t^\beta(x)$ in $(t,x)$) . So the family  $(P_t)_t$ of linear maps defined by
\eqref{markov semi} is well defined on $\mathcal B_b(\overline S_d^+)$, the set of bounded, Borel measurable functions on $\overline S_d^+$.
We only need to show that it gives rise to a semigroup on  $\mathcal B_b(\overline S_d^+)$.

Since the linear hull of the family of exponentials $\{f_u(\xi):= \exp(-\tr(u\xi))\mid u\in S_d^+)$ is dense
in the space of continuous functions vanishing at infinity (and therefore ultimately in $\mathcal B(\overline S_d^+)$,
it suffices to show the semigroup property for the exponential functions $ \xi\mapsto f_u(\xi)$, $u\in S_d^+$. Now, since by Proposition \ref{prop: ric eq} we have that $(\phi,\psi)$ are the {\it unique} solutions
to a system of ordinary differential equations, it follows (from their specific form) that they satisfy the so-called semiflow equations
\begin{align*}
\phi(t+s,u)&=\phi(t,u)+\phi(s,\psi(t,u))\\
\psi(t+s,u)&=\psi(s,\psi(t,u)).
\end{align*}
Hence we can write
\begin{align*}
P_{t+s} f_u(x)&=\int_{\overline S_d^+}f_u(\xi) p_{t+s}(x,d\xi)=e^{-\phi(t+s,u)-\tr(\psi(t+s,u)x)}\\
&=e^{-\phi(t,u)}e^{-\phi(s,\psi(t,u))-\tr(\psi(s,\psi(t,u))x)}=e^{-\phi(t,u)}\int_{\overline S_d^+}f_{\psi(t,u)}(\eta) p_{s}(x,d\eta)\\
&=\int_{\overline S_d^+}e^{-\phi(t,u)-\tr(\psi(t,u),\eta)} p_{s}(x,d\eta)=P_s(P_t f_u(x)).
\end{align*}

It remains to prove the Feller property. By \cite[Proposition III.2.4]{RevuzYor} and using some density argument,
it suffices to show that
\begin{itemize}
\item $P_t f_u(x)\in C_0(S_d^+)$ for all $t\geq 0$, and $u\in S_d^+$ and this can be seen by inspection
of $\psi(t,u)$, which is strictly positive definite, as well.
\item $P_t f_u(x)$ converges pointwise to $f_u(x)$ as $t\rightarrow 0$, which follows
immediately from the continuity of $\phi(t,u)$ and $\psi(t,u)$ in $t$.
\end{itemize}
\end{proof}
A Markov process with transition laws $p_t(x,d\xi)$ on $\overline S_d^+$ is called affine \cite{CFMT}, if eq.~\eqref{affine def mark}
holds. Hence it is obvious that
\begin{lemma}
Wishart processes are affine processes.
\end{lemma}

\subsection*{The generator of a Wishart process}
\begin{lemma}
Let $X$ be a Wishart process on $\overline S_d^+$ with admissible parameters $(p, \beta, \alpha)$. Then the associated semigroup $(P_t)_{t\geq 0}$
has infinitesimal genator $\mathcal A$ acting on $C_b^\infty\subset\mathcal D(\mathcal A)$ as
\begin{equation}\label{eq: form of gen}
\mathcal Af(x)=\frac{1}{2}\sum_{1\leq i,j,k,l\leq d}A_{ijkl}(x)\frac{\partial ^2f(x)}{\partial x_{ij}\partial x_{kl}}+\tr((\beta x+x\beta^\top +2pQ^\top Q)\nabla f(x)),
\end{equation}
where $\nabla f(x)=(\frac{\partial f(x)}{\partial x_{ij}})_{ij}$ and $A_{ijkl}(x)=\left(x_{ik}\alpha_{jl}+x_{il}\alpha_{jk}+x_{jk}\alpha_{il}+x_{jl}\alpha_{ik}\right)$.
\end{lemma}
There are different possible proofs of this fact. By using the fact that $X$ can be realized as solution of a corresponding Wishart SDE $X_t$ starting at $X_0=x$, one could just determine the generator
of $X$ by applying the It\^o-formula or using general results on It\^o-diffusions.
Another, maybe more elegant way is the following. By the very definitions of the Wishart process, we can
calculate the pointwise limit
\begin{align}\nonumber
&\lim_{t\downarrow 0}\frac{P_t f_u(x)-f_u(x)}{t}=\\\nonumber&\quad=f_u(x)\lim_{t\downarrow 0}\frac{e^{-\phi(t,u)-\tr((\psi(t,u)-u)x)}-1}{-\phi(t,u)-
\tr((\psi(t,u)-u)x)}\frac{-\phi(t,u)-\tr((\psi(t,u)-u)x)}{t}\\\nonumber&\quad=f_u(x)
\left(\partial_t\phi(t,u)-\tr(\partial_t\psi(t,u)x)\right)|_{t=0}\\\label{action}&
\quad=-f_u(x) \left(2p\tr(Q^\top Q u)-\tr(-(2 u\alpha u+u \beta+\beta^\top u )x)\right).
\end{align}
The convergence actually holds in sup-norm; this essentially follows from the fact that the pointwise limit
lies in $C_0(\overline S_d^+)$ (see \cite[Proof of Proposition 4.12]{CFMT}). Furthermore, a density argument proves that elements of $C_b^{\infty}$ can be suitably approximated by the linear hull of exponentials $f_u(x)$, $u\in \overline S_d^+$. \footnote{In \cite{CFMT} the rapidly decreasing smooth functions $\mathcal S(\overline S_d^+)$ are used. For the corresponding Stone-Weierstrass Theorem, see \cite[Theorem B.3]{CFMT}} On the other hand it is readily checked that  \eqref{eq: form of gen} evaluated at $f_u(x)$, $u\in \overline S_d^+$ equals eq.~\eqref{action}.
\subsection*{The drift condition}
\begin{theorem}\label{enigmatic theorem}
Let $X$ be a Wishart process on $\overline S_d^+$ with parameters $(p, \beta, \alpha)$, and suppose $\alpha\neq 0$. Then we must have $p\geq \frac{d-1}{2}$.
\end{theorem}
\begin{proof}
Any positive Feller semigroup has an infinitesimal generator $\mathcal A$ which satisfies the strong maximum principle. That is, let $f\in\mathcal D(\mathcal A)$
and $f(x)\geq f(x_0)$ for all $x\in \overline S_d^+$. Then $\mathcal A f(x_0)\leq 0$. (here the following analogy from calculus helps to remember the sign: Let $g$ be a twice differentiable function on an interval $I\subseteq\mathbb R$ which has a local maximum at $x_0$. Then $f''(x_0)\leq 0$. If, in addition, $x_0$ lies in the interior of $I$, then $f'(x_0)=0$. The analogy comes from the fact that the generator of a Feller semigroup has
a principal symbol which is differential operator of second order). In \cite{CFMT} we used the determinant $f(x)=\det(x)$ and (diagonal) boundary points $x_0\in\partial \overline S_d^+$, because $f$ vanishes precisely there.  The theory of \cite{CFMT} is more general than these notes, so it is enough to use \cite[Lemma 4.16 and Lemma 4.17]{CFMT}
to prove the assertion.
\end{proof}
\begin{remark}\rm
Note that when $\alpha=0$ then we have a deterministic motion, because then we have that
\[
\mathbb E[e^{-\tr(uX_t)}\mid X_0=x]=e^{-\tr(u\omega_t^\beta(x))},
\]
i.e.,
$X_t=\omega_t^\beta(x)$. From the Wishart SDE point of view, we clearly have
\[
\dot X(t)=\beta X+X\beta^\top,\quad X_0=x.
\]
In that case, $p$ can be anything but is superfluous.
\end{remark}
\subsection*{Wishart processes are Wishart semimartingales}
So far we did not need to be specific about the realization of Wishart processes as stochastic processes; we only looked at the Markovian transition function. In order to relate Wishart processes and Wishart semimartingales, we consider for each initial state $x$, an associated (to the Wishart transition function) Markov process $X$
on a filtered probability space $(\Omega,\mathcal F,\mathcal F_t,\mathbb P^x)$. Such realizations
exist and are well known. We repeat in the following a little the definitions for Markov transition functions and a canonic construction of the associated stochastic process, which is then called Markov. In the end of the section we prove that every Wishart process is a Wishart semimartingale.

A (suitably measurable) family of probability laws $t\mapsto (p_t(x,d\xi))$  on $\overline S_d^+$, indexed by $t\geq 0$, $x\in \overline S_d^+$,
is called Markovian transition function, if $p_{0}(x,d\xi)=\delta_x(d\xi)$ (the unit mass at $x$) and it satisfies the Chapman-Kolmogorov equations
\begin{equation}\label{eq: CKE}
p_{t+s}(x, A)=\int p_s(\xi,A)p_t(x,d\xi),\quad s,t\geq 0.
\end{equation}
Note that using the function $f_A(x)=1_{A}(x)$ (the indicator function on the Borel set $A$), we can write
\ref{eq: CKE} equivalently in semigroup form
\[
P_{t+s} f_A(x)=P_t(P_s f_A))(x),
\]
where the action $P_t$ is defined above in eq.~\eqref{markov semi}. We have therefore shown
the Chapman-Kolmogorov equations for the continuous functions $f_u$ in Lemma \ref{wish is mark}
and that's enough by some density argument to ensure \eqref{eq: CKE}.

Now by \cite[Theorem 1.1]{EthierKurtz}, for any initial distribution $\nu(d\xi)$ on $\overline S_d^+$
there exists a stochastic process $X$ on a filtered probability space $(\Omega,\mathcal F,\mathcal F_t,\mathbb P)$ whose finite-dimensional distributions fulfil
\begin{align*}
&\mathbb P[X(0)\in A_0, X(t_1)\in A_1,\dots, X(t_n)\in A_n]\\&\qquad=\int_{A_0}\int_{A_1}\dots \int_{A_n} p_{t_n-t_{n-1}}(y_{n-1}, dy_n)p_{t_{n-1}-t_{n-2}}(y_{n-2}, dy_{n-1})\dots p_{t_1}(y_0,
dy_1) \nu(dy_0)
\end{align*}
This construction is ``canonical'' in that $\Omega=(\overline S_d^+)^{[0,\infty)}$ (i.e. the space of all possible paths with values in $\overline S_d^+$), the process
is just given by the projections onto the $t$-th coordinate, that is
\[
X_t(\omega)=\omega(t),\quad \omega\in\Omega
\]
and the sigma algebra is given by the product sigma algebra
\[
\mathcal F=\mathcal B(\overline S_d^+)^{[0,\infty)}.
\]
The filtration is generated by the projections $X_t$:
\[
\mathcal F_t=\sigma(X_s\mid 0\leq s\leq t).
\]

Starting at $\nu(x)=\delta_x(d\xi)$,  where $x\in \overline S_d^+$, we denote the associated probability measure $\mathbb P$ by $\mathbb P^x$. Since $\Omega,\mathcal F$ and $\mathcal F_t$
where independent of the initial law, we have constructed a family of stochastic processes $(\Omega,\mathcal F,\mathcal F_t,\mathbb P^x)$
which satisfy the Markov property for all bounded Borel measurable functions $f$,
\[
\mathbb E^x[f(X_{t+s})\mid \mathcal F_s]=\int_{\overline S_d^+} f(\xi)p_t(X_s,d\xi)=\mathbb E^{X_s}[f(X_t)]
\]
which holds $\mathbb P^x$ a.s., and for all $t,s\geq 0$. Here we use $\mathbb E^x$ to denote the expecation operator with respect to $\mathbb P^x$.

This is equivalent to the more intuitive statement
\[
\mathbb P^x[X_{t+s}\in A\mid \mathcal F_t]=\mathbb P^{X_t} [X_{t+s}\in A].
\]
By Lemma \ref{wish is mark} we also know that $X$ is a Feller process (this is a Markov process with
a Feller semigroup), which implies in view of \cite[Theorem III. 2.7]{RevuzYor} that $X$ admits a cadlag modification. That means for each $x\in \overline S_d^+$, we have that
the probability law $\mathbb P^x$ is actually concentrated on the space of paths which are continuous from the right and have limits from the left. Our aim is
to show that for each $x\in \overline S_d^+$, the process $X_t$ is a Wishart semimartingale. That is continuous, as we know, which will follow a little indirect:
\begin{proposition}\label{prop: Wishart p are Wishart sem}
Let  $X$ be a Wishart process. Then for each $x$, $X$ is a
Wishart semimartingale on $(\Omega,\mathcal F,\mathcal F_t,\mathbb P^x)$.
\end{proposition}
\begin{proof}
Since $X$ is a Feller process, we have by \cite[Proposition VII. 1.6]{RevuzYor}
that for any $f_u(x)=\exp(-\tr(ux))$
\[
M_t^u:= f_u(X_t)-f_u(x)-\int_0^t \mathcal A f_u(X_s)ds
\]
is an $(\mathcal F_t,\mathbb P^x)$--martingale. Hence by \cite[Theorem II.2.42]{JacodShiryaev}
we have that $X$ is a $(\Omega,\mathcal F,\mathcal F_t,\mathbb P)$--semimartingale, associated
to the generator $\mathcal A$. The continuity of $X$ follows from the lack of a jump-component in
the generator (that is the compensator of the jumps of $X$ vanishes). As quadratic variation and drift component are evident from the specific form of the generator, we see that $X$ is a Wishart semimartingale on $(\Omega,\mathcal F,\mathcal F_t,\mathbb P)$.
\end{proof}

\section{Boundary non-attainment}\label{MCKean argument}
Suppose now that $X_0=x$ is positive definite in \eqref{wishart full class}. In view of the standard existence and uniqueness result for
SDEs--the square root is analytic, hence locally Lipschitz on $S_d^+$--there exists a unique strong solution of the Wishart SDE as long as $X_t$ does not hit the  boundary. We call this time
\[
T_x:=\inf\{t>0\mid \det(X)=0\}
\]
the first hitting time of the boundary. Of course when $T_x=\infty$, unique strong solutions of the Wishart SDE
are guaranteed. This is particular the case, when $p$ is large enough.
\begin{theorem}\label{not hitting}
Suppose $p\geq\frac{d+1}{2}$. Then $T_x=\infty$ almost surely.
\end{theorem}
This is a special case of \cite[Corollary 3.2]{pfaffel} but written in
the notation of \cite[chapter 6]{CFMT}. The motivation for this result has been the introductory work of \cite{bru} for Wishart processes
with $Q=I$ and $\beta=0$ (for more detailed comparison with Bru's work, see
 \cite[Proposition 3.1]{pfaffel}). It should be noted that this is a result concerning
the support of Wishart semimartingales, and the existence of strong solutions is a mere by-product of the latter.
For affine jump-diffusions on symmetric cones, the corresponding result is \cite[Proposition 6.1]{CKMT}.

A random time $T:\Omega\rightarrow \mathbb R_+$ is a random variable taking non-negative values.
$T$ is called a stopping time, if the sets $\{T\leq t\}$ are measurable with respect to $\mathcal F_t$.
In our context $T=T_x$ will always be the first hitting time of solutions to Wishart SDEs
of the boundary $\partial \overline S_d^+$. $[0,T)$ is called stochastic interval. A local martingale $M_t$  on
the stochastic interval $[0,T)$ is a stochastic process for which there exists an a.s. strictly increasing sequence $T_n\uparrow T$
such that for each $n$, the stopped process $M_{t\wedge T_n}$ is an $\mathcal F_t$--martingale.

\subsection*{MCKean's argument}
This result on continuous semimartingales is fundamental for the derivation of Theorem \ref{not hitting}. To simplify the setting, we shall from now on assume that $T>0$ almost surely.
This assumption actually holds for $T_x$, because any diffusion started in the interior of some domain needs a strictly positive time
to reach its boundary.

\begin{lemma}
For a {\it continuous} local martingale on $[0,T]$ almost surely
either $\lim_{t\uparrow T} [M_t,M_t]$ exists or we have $\lim\sup_{t\uparrow T} M_t=-\lim\inf_{t\uparrow T}M_t = \infty$.
\end{lemma}
One way of obtaining this result is by performing a time-change $T_t$ on $A:=\{\lim_{t\uparrow T} [M_t,M_t]=\infty\}$ such that
$M_{T_t}$ becomes a continuous local martingale(on $A$) with quadratic variation $t$. Then by L\'evy's characterization of Brownian motion $M_{T_t}$ is a Brownian motion on $A$, hence we just need
to use the pathwise properties of Brownian motion--that a.s. oscillates infinitely often between $-\infty$
and $\infty$, as $t\rightarrow\infty$. The appropriate time change is $T_t:=\inf\{s>0\mid [M_s,M_s]>t\}$.

A stripped-down version of MCKean's argument is the following. A more general formulation may be found in \cite[Proposition 4.3]{pfaffel}:
\begin{proposition}\label{prop MCKean}
Let $Z$ be a continuous adapted stochastic process on a stochastic interval $[0,T)$ such that $Z_0>0$
a.s., and $T:=\inf\{t>0\mid Z_s=0\}$. Suppose $h:\mathbb R_+\setminus\{0\}\rightarrow \mathbb R$
satisfies the following
\begin{enumerate}
\item for $t<T$ we have $h(Z_t)=h(Z_0)+M_t+P_t$, where $P$ is a non-negative process and $M$ is a continous local martingale  on $[0,T)$.
\item $\lim_{z\downarrow 0} h(z)=-\infty$
\end{enumerate}
Then $T=\infty$ a.s.
\end{proposition}
\begin{proof}
As a consequence of the assumptions $h(Z_t)\downarrow -\infty$ as $t\uparrow T$.
Since $P$ is non-negative, we have that  $M_t\downarrow -\infty$ as $t\uparrow T$. But $M$ is a continuous local martingale on $[0,T)$. In view of
the preceding lemma this is only possible, when $T=\infty$.
\end{proof}

Now we shortly sketch the proof of Theorem \ref{not hitting}. All the details can be found
in an old (and unpublished) version of the paper \cite{pfaffel} on \cite[pp. 5--7]{pfaffelold}. They base on a few more Lemmas.
\begin{proof}
We define for $t\in [0, T_x)$
\[
Z_t:=\det(e^{-\beta^\top t}X_t e^{-\beta t}),\quad h(z)=\log(z),
\]
then after application of It\^o's formula \cite[Lemma 4.7]{pfaffel} and some lines of calculations we obtain
\[
h(Z_t)=h(Z_0)+M_t+P_t,
\]
where
\[
M_t=2\int_0^t\sqrt{\tr(X_s^{-1} Q^\top Q)}dW_s
\]
and
\[
P_t=\int_0^t\tr((2p-(d+1)Q^\top Q)X_s^{-1})ds,
\]
where $W$ is a one-dimensional standard Brownian motion. Hence $M$ is a continuous local martingale on $[0,T_x]$ and $P$ is non-negative. Proposition
\ref{prop MCKean} can be applied and yields $T_x=\infty$.
\end{proof}
\subsection*{Hitting the boundary}
The following shows that Theorem \ref{not hitting} does not hold under weaker conditions:
\begin{lemma}\label{lem definite hit}
Let $\beta, Q$ be $d\times d$ matrices, and suppose $Q\neq 0$. When $p<\frac{d-1}{2}$, then there exists $x\in \overline S_d^+$ such that $T_x<\infty$ with positive probability.
\end{lemma}
\begin{proof}
Assume, for a contradiction, that for all $x\in \overline S_d^+$ we have that $T_x=\infty$. By Lemma \ref{wish sde is wishart semi} any solution $X_t$ of the Wishart SDE is a Wishart semimartingale. And by Lemma \ref{lem wish semi is wish dist} we have that $X_t\sim \Gamma(p, \omega_t^\beta(x);\sigma_t^\beta(\alpha))=:p_t(x,d\xi)$ where $\alpha=Q^\top Q$. By definition $(p_t(x,d\xi))_{t\geq 0, x\in \overline S_d^+}$ is a Wishart process. By Theorem \ref{enigmatic theorem} we must have  $p\geq\frac{d-1}{2}$, a contradiction.
\end{proof}
For a similar and partially stronger result see Lemma \ref{stronger lem definite hit} below.

In the case that $\beta=0$ and $Q=I$, \cite[Theorem 1.4]{matsumoto} asserts that the boundary is hit in finite time, when $p\in (\frac{d-1}{2},\frac{d+1}{2})$. A similar result
including general $\beta$ or $Q\neq 0$ seems not to be known yet. However, we conjecture
\begin{conjecture}
Let $\beta, \,Q$ be arbitrary $d\times d$ matrices, and let $p<\frac{d+1}{2}$. Any solution of the Wishart SDE with initial condition $X_0=x\in S_d^+$
hits the boundary in finite time, that is $\mathbb P(T_x<\infty)>0$.
\end{conjecture}
We further conjecture
\begin{conjecture}
Let $\beta, \,Q$ be arbitrary $d\times d$ matrices, and let $p\leq \frac{d-1}{2}$. Any solution of the Wishart SDE with initial condition $X_0=x\in S_d^+$
hits the boundary almost surely, that is $\mathbb P(T_x<\infty)=1$.
\end{conjecture}
\section{Changing the drift}
Let $(\Omega,\mathcal F,\mathcal F_t,\mathbb P)$ be a filtered probability space satisfying the usual conditions, which supports an $\mathcal F_t$--Brownian motion. Girsanov transformations are tools to derive solutions to SDEs as follows. We consider for a moment the one-dimensional case. Let $X_t$ be a solution of
\[
dX_t=b(X_t)dt+\sigma(X_t)dW_t,\quad X_0=x.
\]
Suppose we  actually seek to solve such an equation for an alternative drift $\tilde b(\cdot)$. If $\sigma$ is invertible, we can rewrite the above equation as
\[
dX_t=\tilde b(X_t)dt+ (b(X_t)-\tilde b(X_t))dt+\sigma(X_t)dW_t=\tilde b(X_t)dt+\sigma(X_t)\left(\frac{b(X_t)-\tilde b(X_t)}{\sigma(X_t)}dt+dW_t\right).
\]

In the following we abbreviate $\gamma_t:=\frac{b(X_t)-\tilde b(X_t)}{\sigma(X_t)}$. If we can show that under a new probability measure $\mathbb Q$, the process
$ \int_0^t\gamma_s ds+W_t$ is a Brownian motion, then we have achieved our goal. Note the weak character of this solution: The Brownian motion is not given in advance, but constructed from the pair $(X_t, W_t)$.

What we have outlined above is indeed possible; it is a consequence of Girsanov's theorem, which asserts that if
\[
Z_t:=\mathcal E(-\int_0^t\gamma_s dW_s)=e^{-\int_0^t \gamma_s dW_s-\frac{1}{2}\int_0^t \vert \gamma_s\vert^2 ds}
\]
is a martingale on $[0,T]$, then $\mathbb Q$ defined as
\[
d\mathbb Q=\mathcal E(-\int_0^T\gamma_s dW_s) d\mathbb P
\]
is a probability measure equivalent to $\mathbb P$, and $\int_0^t\gamma_s ds+ W_t$ is a $\mathbb Q$--Brownian motion on $[0,T]$. The essential problem therefore is to show the martingale property of $(Z_t)_{t\leq T}$. That can be quite tricky.

Bru \cite{bru} used the Girsanov theorem to derive solutions of Wishart SDEs with nonzero linear drift
$\beta$ from SDEs with constant drift only. In special cases she derives solutions until the first time the eigenvalues of the process collide.
The respective time of collision is not dealt with in her work when $\beta\neq 0$; recent work
elaborates on this issue, see Graczyk and Malecki \cite{GraczykMalecki}.

We have already shown the existence of solutions in the preceding chapter when $\beta \neq 0$
under the more stringend condition $p\geq \frac{d+1}{2}$.  So we do not need the Girsanov theorem
to create new solutions, and we also never had to care about the collision of eigenvalues.  But what we can do is to relate solutions with respect to different
drift parameters to each other:
\begin{theorem}
Suppose $X_t$ is a solution of a Wishart SDE with parameters $(p, \;\beta,\;Q)$, where
$Q$ is invertible, and let $X_0=x\in S_d^+$. For $p^\mathbb Q\in\mathbb R$ and a $d\times d$ matrix $\beta^\mathbb Q$ we set
\begin{equation}\label{def: gamma}
\gamma_t:=\sqrt{X_t}((\beta^\top-(\beta^\mathbb Q)^\top)Q^{-1}+ (p-p^\mathbb Q)\sqrt{X_t}^{-1}Q^\top
\end{equation}
and
\begin{equation}\label{def: stoch exp}
Z_t:=\mathcal E\left(-\int_0^t\tr(\gamma_t dB_t) \right).
\end{equation}
If $\min (p,p^\mathbb Q)\geq \frac{d+1}{2}$,  then $Z_t$ is a martingale  on $[0,T]$, and $dB^\mathbb Q_t:=\gamma_t+ B_t$ is a $\mathbb Q$--Brownian motion
on  $[0,T]$. Furthermore $X_t$ satisfies
the Wishart SDE with parameters $(p^\mathbb Q,\beta^\mathbb Q, Q)$ under $\mathbb Q$.
\end{theorem}
\begin{remark}\rm
\begin{itemize}
\item On the level of Wishart semimartingales, the result translates in a statement
 concerning their absolute continuity, see \cite{Cher II}.
\item The theorem bases on the fact that under the condition $\min (p,p^\mathbb Q)\geq \frac{d+1}{2}$
the respective Wishart semimartingales do not attain the boundary $\partial \overline S_d^+$ in finite time,
see Theorem \ref{not hitting}. Note also: it is impossible to define $\gamma_t$ unless $X_t\in S_d^+$
for all $t\leq T$.
\item The best known sufficient criterium for $\mathcal E(-\int_0^t\gamma_s dB_s)$ to be a martingale
is provided by Novikov's condition,
\[
\mathbb E[e^{\frac{1}{2}\int_0^T \|\gamma_t\|^2dt}]<\infty.
\]
This condition is hard to check in our context. Furthermore, it fails, in general. In fact, for the particular case $Q=I$, $\beta=0,\beta^\mathbb Q=I$, $p=p^\mathbb Q$ we have that
\[
\mathbb E[e^{\frac{1}{2}\int_0^T \|\gamma_t\|^2dt}]=
\mathbb E[e^{\frac{1}{2}\int_0^T \tr(X_t)dt}]
\]
which is infinite for sufficiently large $T$. To see this
we interpret it as the exponential moment of a new stochastic process (a so-called affine process)
$(X_t, Y_t=\tr(X_t))$ on $\overline S_d^+\times \mathbb R_+$ whose moment generating equals
\begin{equation}\label{mom ex}
\mathbb E[e^{\frac{1}{2}\int_0^T Y_tdt}]=e^{\phi(t)+\tr(x\psi(t))},
\end{equation}
where $\psi(t)$ satisfies the ODE
\[
\partial_t\psi(t)=2\psi(t)^2+\frac{I}{2},\quad \psi(0)=0.
\]
This is a matrix Riccati differential equation which has explosion in finite time (say at $t_+>0$), and by the positivity of $x$
we have that $\tr(x\psi(t))\uparrow \infty$ as $t\uparrow t_+$. It follows that the moment \eqref{mom ex} explodes.

For more information on the technique of enlargement of the state space and calculation of the moment
generating function of affine processes, see for instance \cite[Proof of Theorem 4.1]{ADPTA}.
\end{itemize}
\end{remark}

\begin{proof}
We start with the second part. Under the premise that $Z$ is a true martingale, the conclusion of
Girsanov's theorem holds and we obtain
\begin{align}\label{eq under P}
dX_t&=\sqrt {X_t} dB_t Q+Q^\top dB_t^\top \sqrt {X_t}+(2p\,Q^\top Q+\beta X_t+X_t\beta^\top )dt\\
&=\sqrt{X_t}(\gamma_t dt+dB_t)Q+Q^\top (\gamma_t^\top dt+dB_t^\top)\sqrt X_t +
(2p^\mathbb Q\,Q^\top Q+\beta^\mathbb Q X_t+X_t(\beta^\mathbb Q)^\top )dt\\\label{eq under P}
&=\sqrt {X_t} dB^\mathbb Q_t Q+Q^\top d(B^\mathbb Q)_t^\top \sqrt {X_t}+(2p^\mathbb Q\,Q^\top Q+\beta^\mathbb Q X_t+X_t(\beta^\mathbb Q)^\top )dt
\end{align}
and therefore $X_t$ is a solution of the Wishart SDE on $[0,T]$ with new parameters $(p^\mathbb Q,\beta^\mathbb Q, Q)$ under the measure $\mathbb Q$.

It remains to show that $Z_t$ is a true martingale. We use the exact arguments as provided by
the proof of \cite[Theorem 1]{Cher I} but adapted to our matrix-valued setting.

Since $p\geq \frac{d+1}{2}$, we have a well defined positive definite solution $X_t$
of the Wishart SDE \eqref{eq under P} (subject to $X_0=x$) in view of Theorem \ref{not hitting} and therefore the process $\gamma (X_t)$ of eq.~\eqref{def: gamma} is well defined
on $0\leq t\leq T$. The stochastic exponential $Z_t$ given by \eqref{def: stoch exp} is a strictly positive local martingale, hence it is a supermartingale. To show that it is a true martingale, it suffices to prove\footnote{Every positive local martingale is a supermartingale and every supermartingale with constant expectation is a martingale. See \cite[p.123]{RevuzYor}.} that
\begin{equation}\label{con: terminal}
\mathbb E[Z_T]=1.
\end{equation}
Quite similarly, there also exists a solution $\widetilde X_t$
of the Wishart SDE
\[
d\widetilde X_t=\sqrt {\widetilde X_t} dB_t Q+Q^\top d B_t^\top \sqrt {\widetilde X_t}+(2p^\mathbb Q\,Q^\top Q+\beta^\mathbb Q \widetilde X_t+\widetilde X_t(\beta^\mathbb Q)^\top )dt
\]
subject to the same initial condition $\widetilde X_t=x$ (note here: we use the desired new drift parameters with $Q$ superscripts,
but the SDE is driven by the original Brownian motion $B$). This process serves as auxiliary process to show
condition \eqref{con: terminal}. We also can define $\gamma_t(\widetilde X_t)$
exactly as in \eqref{def: gamma}, but using $\widetilde X_t$ instead of $X_t$.

We introduce the two sequences of stopping times
\[
\tau_n=\inf\{t>0\mid \|\gamma(X_t)\|\geq n\}\wedge T
\]
and
\[
\widetilde{\tau}_n=\inf\{t>0\mid \|\gamma(\widetilde X_t)\|\geq n\}\wedge T.
\]
These are increasing sequences satisfying
\begin{equation}
\lim_{n\rightarrow\infty}\mathbb P(\tau_n=T)=\lim_{n\rightarrow\infty}\mathbb P(\widetilde{\tau}_n=T)
\end{equation}
because we use the convention that the infimum of an empty set is $+\infty$. For each $n\geq 1$
we define the process
\begin{equation}
\gamma_t^n:=\gamma(X_t)1_{t\leq \tau_n},\quad t\in[0,T].
\end{equation}
By construction $\int_0^t\|\gamma_s^n(X)\|^2\leq n^2 t$, and therefore Novikov's condition
\[
\mathbb E[e^{\frac{1}{2}\int_0^T\|\gamma^n_s(X)\|^2ds}]<\exp(n^2 T/2)
\]
which let us conclude that
\[
Z_t^n=
\mathcal E\left(-\int_0^t\tr(\gamma^n_s dB_s) \right)
\]
is a martingale and
$d\mathbb Q^n:=Z_T^n d\mathbb P$ defines a probability measure
equivalent to $\mathbb P$ for which
$B^n_t=\int_0^t \gamma_s^n(X)ds+B_t $ is a $d\times d$
standard Brownian motion on $[0,T]$. Furthermore, for each $n$,
the stopped process $X_t^n:=X_{t\wedge \tau_n}$ have the same law under $\mathbb Q^n$
as the stopped processes $\widetilde X_t^n=\widetilde X_{t\wedge \widetilde{\tau}_n}$
under $\mathbb P$. We therefore have
\begin{align*}
\mathbb E[Z_T]&=\lim_{n\rightarrow\infty} \mathbb E[Z_T^n 1_{\tau_n=T}]\\
&=\lim_{n\rightarrow\infty} \mathbb E^{\mathbb Q^n}[ 1_{\tau_n=T}]=
\lim_{n\rightarrow\infty} \mathbb  Q^n(\{\tau_n=T\})\\
&=\lim_{n\rightarrow\infty} \mathbb  P(\{\widetilde{\tau}_n=T\})=1,
\end{align*}
where the first identity follows from monotone convergence (which is applicable because the sets $\tau_n=T$
are increasing in $n$, and $Z_T^n$ is a constant sequence along this sequence; hence the sequence
$Z_T^n 1_{\tau_n=T}$ is a monotonically increasing one).
\end{proof}

\section{On the existence of Wishart distributions}
In this section we provide some results concerning the existence of Wishart distributions and their densities.
To this end, we introduce some further notation. Let $a\in\mathbb R$
and $k\in\mathbb N_0$. The hypergeometric coefficient $(a)_k$ is defined as
\[
(a)_k:=\begin{cases}1,\quad\textrm{  if  } k=0\\ a(a+1)\cdot \dots\cdot (a+k-1),\textrm{  otherwise  }\end{cases}.
\]
Let $\kappa=(\kappa_1,\dots,\kappa_d)\in\mathbb N_0^d$ be a multi-index with length $\vert\kappa\vert=\kappa_1+\dots+\kappa_d$. The generalized ($d$--dimensional) hypergeometric  coefficient $(p)_{\kappa}$ is given by
\[
(p)_{\kappa}=\prod _{j=1}^d\left(p-\frac{1}{2}(j-1)\right)_{\kappa_j},
\]
see for instance \cite[p.~30]{GuptaNagar01}. $C_{\kappa}:S_d\rightarrow \mathbb R$ shall denote the zonal polynomial of order $\kappa$, where $\kappa\in\mathbb N_0^d$. There are several equivalent definitions, for instance $C_\kappa(\xi)$ equals the $\kappa$th component of
$(\tr(\xi))^k$, (see \cite[Definition 1.5.1]{GuptaNagar01}). Hence
\[
(\tr(\xi))^k=\sum_{\vert\kappa\vert=k} C_\kappa(\xi).
\]
A more abstract definition \cite[p. 234]{Faraut} is that
\begin{equation}
C_{\kappa}(\xi):=\omega_\kappa\int_{k\in SO(d)}\Delta_{\kappa}(k\cdot \xi)dk
\end{equation}
where $dk$ is the normalized unique Haar measure on the special orthogonal group $SO(d)$, and $\omega_\kappa$ is some normalizing constant.

\begin{proposition}\label{lemma: existence wishart dist}
Let $p\in \Lambda_d$, $\sigma\in \overline S_d^+$ and $\omega\in \overline S_d^+$. We have:
\begin{enumerate}
\item \label{ex wis dist 0} If $2p\in \mathbb N$ and if $\rank(\omega)\leq 2p$, then $\Gamma(p,\omega;\sigma)$ exists.
\item \label{ex wis dist 1} If $p\geq \frac{d-1}{2}$, then the right side of \eqref{FLT Mayerhofer Wishart} is the Laplace transform of a probability measure $\Gamma(p, \omega;\sigma)$ on $\overline S_d^+$.
\item \label{ex wis dist 2} In particular, if $p>\frac{d-1}{2}$ and if $\sigma$ is invertible, then the density of $\Gamma(p, \omega;\sigma)$ exists and is given by

    \begin{align}\label{densities for wishart dist}
F(p, \omega;\sigma)(\xi)&:=(\det\sigma)^{-p}\,e^{-\tr(\sigma^{-1}\xi+\sigma a)}(\det \xi)^{p-\frac{d+1}{2}}\\\nonumber
&\quad\times \left(\sum_{m=0}^\infty\sum_{\vert\kappa\vert=m}\frac{C_\kappa(\sqrt{a}\xi \sqrt{a})}{m!(p)_{\kappa}}\right)\frac{1_{\overline S_d^+}(\xi)}{\Gamma_k(p)},
\end{align}
where we have set $a=a(\omega):=\sigma^{-1}\omega\sigma^{-1}$, $q:=q(\sigma)=\sqrt{\sigma}$.

    \item \label{ex wis dist 3} If $\sigma$ is degenerate, $\Gamma(p, \omega;\sigma)$ is not absolutely continuous with respect to the Lebesgue measure on $\overline S_d^+$.
\end{enumerate}
\end{proposition}
\begin{proof}
Statement \label{ex wis dist 0} is proved by summing up squares of normally distributed $\mathbb R^d$-valued random variables, see section \ref{sec: intro}.

Note that if $\sigma\in S_d^+$, our definition of non-central Wishart distribution is related to the one of
\cite{Letac08} in that $\Gamma(p,\omega;\sigma)=\gamma(p,\sigma^{-1}\omega\sigma^{-1};\sigma)$, the latter being
called ''general non-central Wishart distribution" in \cite{Letac08}. Hence statement \ref{ex wis dist 2} is  a
consequence of \cite[p.\ 1400]{Letac08}.

Now for each $\varepsilon>0$ we regularize $\sigma$, $a$ and $p$ by setting
\[
\sigma_\varepsilon:=\sigma+\varepsilon I, \quad a_\varepsilon:=(\sigma+\varepsilon I)^{-1}\omega (\sigma+\varepsilon I)^{-1}, \quad p_\varepsilon=p+\varepsilon.
\]
Then for each $\varepsilon>0$, we pick $X_\varepsilon$, an $\overline S_d^+$ valued random variable according to
\cite[Proposition 2.3]{Letac08} such that
\[
X_\varepsilon\sim\Gamma(p_\varepsilon,\omega;\sigma_\varepsilon) (=\gamma(p_\varepsilon,a_\varepsilon;\sigma_\varepsilon)).
 \]
Letting $\varepsilon\rightarrow 0$ and using L\'evy's continuity theorem, we figure
that
\[
 \left(\det(I+\sigma u)\right)^{-p}e^{-\tr(u(I+\sigma
u)^{-1}\omega)}=\lim_{\varepsilon\rightarrow 0} \left(\det(I+\sigma_\varepsilon u)\right)^{-p_\varepsilon}e^{-\tr(u(I+\sigma_\varepsilon
u)^{-1}\omega)}
\]
must be the Laplace transform of some random variable $X\sim\Gamma(p,\omega;\sigma)$, to which
$X_\varepsilon$ converges in distribution as $\varepsilon\rightarrow 0$.
This settles part \ref{ex wis dist 1}. Finally, we consider assertion \ref{ex wis dist 3}. Assume, by contradiction,
that $\Gamma(p,\omega;\sigma)$ has a Lebesgue density, for some $\sigma$ of rank $r<d$.
Let $X$ be an $\overline S_d^+$--valued random variable distributed according to $\Gamma(p,\omega;\sigma)$. Since linear transformations do not affect the property of having a density
and since the non-central Wishart family is invariant under linear transformations (this is easy to check), we may without loss of generality assume that $\sigma=\diag(0,I_r)$, where $I_r$ is the $r\times r$ unit matrix. Consider the projection
\[
\pi_r: x=(x_{ij})_{1\leq i,j\leq d}\mapsto \pi_r(x):=(x_{ij})_{1\leq i,j\leq r}.
\]
A simple algebraic manipulation yields that the Laplace transform of $\pi_r(X)$ equals
\[
e^{-\tr(\pi_r(\omega)v)}, \quad v\in S_{r}^+,
\]
which is the Laplace transform of the unit mass concentrated at $\pi_r(\omega)$. But the pushforward of a measure with density under a projection must have a density again. This yields the deserved contradiction.
\end{proof}
A very important consequence of this statement in combination of
the results of section \ref{sec: relations} is the following existence result
\begin{corollary}
For all $p\geq \frac{d-1}2$,  $\alpha\in \overline S_d^+$ and $d\times d$ matrices $\beta$,
Wishart processes with Wishart transition function with parameters $(\alpha, p,\beta)$ exist.
Therefore all Wishart semimartingales with the same parameters, starting at $x\in \overline S_d^+$ exist.
Similarly, for all $Q$ with $Q^\top Q=\alpha$ and for all $x\in \overline S_d^+$ the Wishart SDE admits global weak solutions.
\end{corollary}
\begin{proof}
Proposition \ref{lemma: existence wishart dist} \ref{ex wis dist 0} allows a well defined Wishart transition
function of the form \eqref{def: Wishart}. This transition function is Markovian by Lemma \ref{wish is mark}.
Now we can combine Proposition \ref{prop: Wishart p are Wishart sem}, Lemma \ref{wish sde is wishart semi} to obtain the remaining assertions.
\end{proof}

\section{A rank condition for non-central Wishart distributions}
Not for all triples $(p,\omega,\sigma)\in\mathbb R_+\times \overline S_d^+\times \overline S_d^+$
Wishart distributions $\Gamma(p,\omega;\sigma)$ exist. \cite{MayerhoferExistenceNonCentral}
shows the following:
\begin{theorem}\label{th: maintheorem}
Let $d\in\mathbb N$, $p> 0$, $\omega \in \overline S_d^+$.  Suppose $\sigma\in \overline S_d^+$ is invertible.
If the right side of \eqref{FLT Mayerhofer Wishart} is the Laplace transform of a non-trivial probability
measure $\Gamma(p,\omega;\sigma)$ on $\overline S_d^+$, then $p\in\Lambda_d$ and
$\rank(\omega)\leq 2p+1$.
\end{theorem}
This result contradicts the preceding characterization of Letac and Massam \cite{Letac08},
where no constraint on the non-centrality parameter had been imposed, which we call here rank-condition. Motivated
by \cite{MayerhoferExistenceNonCentral}, Letac and Massam \cite{Letac11} deliver very recently an even stronger result which uses different methods, and fully characterizes the existence and non-existence
of the non-central Wishart family (see Theorem \ref{th: letac11} below).

Theorem \ref{th: maintheorem}  uses very nicely the construction and properties of Wishart processes, but also elementary arguments, such as L\'evy's continuity theorem. The latter allows to conclude, by using
the characterization of central Wishart distributions, that $p\in \Lambda_d$. The proof for the rank condition is indirect; we assume, for a contradiction, the existence of a single Wishart distribution which violates the rank condition. We then use the exponential family of the latter to construct a whole family of Wishart laws, which determine a Wishart process on $\overline S_d^+$. That, in turn, ultimately violates the drift condition of Theorem \ref{enigmatic theorem}.
We start with a few lemmas.

Let $(p,\omega,\sigma)\in \mathbb R_{++}\times \overline S_d^+\times S_d^+$ such that
$\mu:=\Gamma(p,\omega;\sigma)$ is a probability measure, that is,
eq.~\eqref{FLT Mayerhofer Wishart} holds. The domain of its
moment generating function is defined as
\begin{equation*}
D(\mu):=\{u\in S_d\mid \mathcal L_\mu(u):=\int_{\overline S_d^+}e^{-\langle
u,\xi\rangle}\mu(d\xi)<\infty\},
\end{equation*}
which is the maximal domain to which the Laplace transform, originally defined for $u\in \overline S_d^+$ only, can be extended.
It is well known that $D(\mu)$ is a convex (hence connected) set, and we also know that
$\overline S_d^+\subset D(\mu)$. Clearly $(I+\sigma u)$ is invertible if and only if the (symmetric) matrix
$(I+\sqrt{\sigma} u \sqrt{\sigma})$ is non-degenerate. Using these facts and the defining equation
\eqref{FLT Mayerhofer Wishart} we infer that
\begin{equation}\label{equalityofsets}
D(\mu):=\{u\in S_d\mid (I+\sqrt{\sigma}
u\sqrt{\sigma})\in S_d^+\}=-\sigma^{-1}+S_d^+,
\end{equation}
and therefore $D(\mu)$ is even open. Accordingly, the natural exponential family of
$\mu$ is the family of probability measures\footnote{In order to avoid confusions with calculations in the proof of the upcoming proposition, we change here from $u$ notation to $v$, because
$u$ denotes the Fourier-Laplace variable in this paper.}
\[
F(\mu)=\left.\left\{\frac{\exp(v\xi) \mu(d\xi)}{\mathcal L_\mu(v)}\right| v\in -\sigma^{-1}+S_d^+\right\}.
\]
We start by stating some key properties of Wishart distributions \footnote{Some related properties can be found in
Letac and Massam \cite{Letac08}, but in a different notation. Letac and Massam use instead of
 $\Gamma(p,\omega;\sigma)$ the parameterized family $\gamma(p,a;\sigma)$, where
$\omega$ is replaced by $a:=\sigma^{-1}\omega \sigma^{-1}$. Accordingly
\eqref{FLT Mayerhofer Wishart} can be written in the form
\begin{equation}\label{FLT Letac}
\mathcal L (\gamma(p,a;\sigma))(u)= \left(\det(I+\sigma u)\right)^{-p}e^{-\tr(u(I+\sigma
u)^{-1}\sigma a\sigma)},\quad u\in \overline S_d^+.
\end{equation} Note that this requires $\sigma$ to be invertible.}:
\begin{proposition}\label{proplemma}
\begin{enumerate}
\item \label{firstelemprop} Let $p\geq 0,\,\omega\in \overline S_d^+$. Suppose $X$ is an $\overline S_d^+$-valued random variable distributed according to $\Gamma(p,\omega; I)$. Let $q\in \overline S_d^+$ and set $\sigma:=q^2$.
Then $q X q \sim \Gamma(p,q\omega q;\sigma)$\footnote{Expressed in geometric language, we say that the pushfoward of $\Gamma(p,\omega; I)$ under the map
$\xi\mapsto q \xi q$ equals $\Gamma(p,q\omega q;\sigma)$}. In particular, $\Gamma(p,\omega; I)$ exists if and only if $\Gamma(p, q \omega q;\sigma)$ exists.
\item  \label{secondelemprop} Let $p\geq 0,\, \sigma \in S_d^+$ and $\omega\in \overline S_d^+$ such that $\mu:=\Gamma(p,\omega;I)$ is a probability measure.
For $v=\sigma^{-1}-I$ we have that
\begin{equation}\label{claim 2}
\frac{\exp(v\xi) \mu(d\xi)}{\mathcal L_\mu(v)}\sim \Gamma(p, \sigma\omega\sigma;\sigma).
\end{equation}
Conversely, if  $\Gamma(p, \sigma\omega\sigma;\sigma)$ is a well defined probability measure, so is
$\mu$, and \eqref{claim 2} holds. In particular, we have that the exponential family generated by $\mu$ is a Wishart family and equals
\[
F(\mu)= \{\Gamma(p,\sigma\omega\sigma, \sigma)\mid \sigma\in S_d^+,\quad \sigma^{-1}-I\in D(\mu)\}.
\]
\item \label{proppoint 3} Suppose that $\Gamma(p,\omega_0;\sigma_0)$ is a probability measure, for $p\geq 0$ and $\omega_0,\sigma_0\in S_d^+$. Then we have:
\begin{enumerate} \item \label{proppoint 3a} $\Gamma(p,t \omega_0;\sigma_0)$ is a probability measure for each $t>0$.
\item \label{proppoint 3b} If, in addition, $\omega_0$ is invertible, then $\Gamma(p, \omega;\sigma)$ is a probability measure for each $\omega\in \overline S_d^+$, $\sigma\in \overline S_d^+$.
\end{enumerate}
\end{enumerate}
\end{proposition}
\begin{proof}
Let $\mathbb E$ be the corresponding expectation operator. By repeated use of the cyclic property of the trace and by the product formula for the determinant, we have
\begin{align*}
\mathbb E[e^{-\langle u, q X q\rangle}]&=\mathbb E[e^{-\langle q u q, X\rangle}]=\det(I+ quq)^{-1}\exp(-\tr( quq(I+quq)^{-1}\omega))\\
&=\det(I+ \sigma u)^{-1}\exp(-\tr( uq(I+quq)^{-1}q^{-1} q \omega q))\\&=\det(I+ \sigma u)^{-1}\exp(-\tr( u (I+\sigma u)^{-1} q \omega q)),
\end{align*}
which proves assertion \ref{firstelemprop}. Next we show \ref{secondelemprop}. We note first, that by
\eqref{equalityofsets} we have that $v=\sigma^{-1}-1\in D(\mu)$. Hence exponential tilting is admissible. Furthermore, we have
\begin{equation}\label{eq: prop FLT}
\int _{\overline S_d^+}e^{-\langle u+v,\xi\rangle}\Gamma(p,\omega; I)(d\xi)=\det(I+(u+v))^{-p}\exp(-\tr((u+v)(I+u+v)^{-1}\omega)),
\end{equation}
and setting $v=\sigma^{-1}-I$ we obtain
\[
I+u+v=\sigma^{-1}(I+\sigma u).
\]
Hence the first factor on the right side of eq.~\eqref{eq: prop FLT} is proportional to $\det(I+\sigma u)^{-p}$. It remains to show
that
\begin{equation}\label{eq: what we would like to have}
-\tr((u+v)(I+u+v)^{-1}\omega)=c+\tr(u(I+\sigma u)^{-1} \sigma\omega\sigma)
\end{equation}
for some real constant $c$, because then the right side
of \eqref{eq: prop FLT} is proportional to the Laplace transform of $\Gamma(p,\sigma\omega\sigma;\sigma)$. To this end, we do some elementary algebraic manipulations:
\begin{align*}
-(u+v)(I+u+v)^{-1}\omega&=-(u-I+\sigma^{-1})[\sigma^{-1}(I+\sigma u)]^{-1}  \omega\\&=-(-I+\sigma^{-1}+u)(\sigma^{-1}+u)^{-1}\omega\\&=-\omega+(\sigma-\sigma)\omega+(\sigma^{-1}+u)^{-1}\omega\\&=(\sigma-I)\omega-\sigma(\sigma^{-1}+u)(\sigma^{-1}+u)^{-1}\omega+(\sigma^{-1}+u)^{-1}\omega\\&=(\sigma-I)\omega-\sigma u (\sigma^{-1}+u)^{-1}\omega\\&=(\sigma-I)\omega-\sigma u (I+\sigma u)^{-1}\sigma\omega.
\end{align*}
We set now $c:=\tr((\sigma-I)\omega)$ which is the real number we talked about before. Taking trace and performing cyclic permutation inside,
we obtain \eqref{eq: what we would like to have}, and therefore the idendity \eqref{claim 2} is shown. The assertion concerning the exponential family follows by the very definition of the latter.

We may therefore proceed to \ref{proppoint 3} which is proved by
repeatedly applying \ref{firstelemprop} and \ref{secondelemprop}: Let
$\Gamma(p,\omega_0;\sigma_0)$ be a probability measure. Then by
\ref{secondelemprop}, also $\Gamma(p,\sigma_0^{-1}\omega_0\sigma_0^{-1}; I)$ is one. Let $q_1$ such that
$q_1^2=\sigma_1\in S_d^+$. We may write
$\Gamma(p,\sigma_0^{-1}\omega_0\sigma_0^{-1};I)=\Gamma(p,q_1^{-1}(q_1 \sigma_0^{-1}\omega_0\sigma_0^{-1}q_1)q_1^{-1};I)$, and by
applying \ref{firstelemprop}, we obtain the pushforward measure
$\Gamma(p,q_1 \sigma_0^{-1}\omega_0\sigma_0^{-1}q_1;\sigma_1)$. By \ref{secondelemprop} we
have that $\Gamma(p, q_1^{-1}\sigma_0^{-1}\omega_0\sigma_0^{-1} q_1^{-1};I)$ is a probability measure
as well, and once again by
 \ref{secondelemprop} we infer that for all $\sigma\in S_d^+$, $\Gamma(p, \sigma q_1^{-1}\sigma_0^{-1}\omega_0\sigma_0^{-1} q_1^{-1}\sigma,\sigma)$
 is a probability.  We use this fact to prove both parts of the assertion. Without loss of generality we assume that
$\sigma$ is non-degenerate, because in the case $\sigma\in\partial \overline S_d^+$ we may invoking L\'evy's continuity theorem\footnote{Strictly speaking, L\'evy's continuity theorem applies to characteristic functions. However, in the Wishart case,
the right side of \eqref{FLT Mayerhofer Wishart} can even be extended to even the Fourier-Laplace transform with ease, and by preserving its functional form.}.
Setting $q_1=1/\sqrt{t} I$ and  $\sigma=\sigma_0$, we see that \ref{proppoint 3a} holds. For $\omega_0\in S_d^+$ we choose $q_1\in
\overline S_d^+$ such that $q_1^{-1} \sigma_0^{-1}\omega_0\sigma_0^{-1} q_1^{-1}=\sigma^{-1}\omega\sigma^{-1}$, which allows to conclude \ref{proppoint 3b}.
\end{proof}
Next, we restate the characterization of the central
Wishart laws by using \cite{PeddadaRichards91}:
\begin{theorem}\label{th: mainproposition}
Let $d\geq 2$, $\sigma \in S_d^+$ and $p\geq 0$. The following
are equivalent:
\begin{enumerate}
\item \label{central1} $\det(I+\sigma u)^{-p}$ is the Laplace transform of a probability measure $\Gamma(p,\omega;\sigma)$ on $\overline S_d^+$.
\item \label{central2} $p\in \Lambda_d$.
\end{enumerate}
\end{theorem}

We are prepared to deliver our proof of Theorem \ref{th: maintheorem}:
\begin{proof}
Let $p>0$ such that for some $\omega_0\in \overline S_d^+$, $\sigma\in S_d^+$, the right
side of \eqref{FLT Mayerhofer Wishart} is the Laplace transform of a
non-trivial probability measure $\Gamma(p,\omega_0;\sigma)$. By
Proposition \ref{proplemma} \ref{proppoint
3a}, we have that $\Gamma(p,\omega_0/n;\sigma)$ is a probability
measure for each $n\in\mathbb N$. Letting $n\rightarrow\infty$ and
invoking L\'evy's continuity theorem, we obtain that
$\Gamma(p;\sigma)$ is a probability measure. But then by the characterization of central Wishart laws,
Theorem \ref{th: mainproposition} \ref{central2}, we have that $p\in\Lambda_d$.

Let now $p_0\in \Lambda_d\setminus [\frac{d-1}{2},\infty)$, and let
us assume, by contradiction, that there exist $(\omega_0,\sigma)\in
\overline S_d^+\times S_d^+$, $\rank(\omega_0)>2p_0+1$ such that
$\Gamma(p_0,\omega_0;\sigma)$ is a probability measure. Pick now
$\omega_1\in \overline S_d^+$ such that $\omega^*:=\omega_1+\omega_0$ has
$\rank(\omega^*):=\rank(\omega_1)+\rank(\omega_0)=d$, and set
$p_1:=\frac{d-\rank(\omega_0)}{2}$. By construction $2p_1=
\rank(\omega_1)$, and
$p_1\in\Lambda_d\setminus[\frac{d-1}{2},\infty)$. Hence Proposition
\ref{lemma: existence wishart dist} \ref{ex wis dist 0} implies the
existence of a non-central Wishart distribution
$\Gamma(p_1,\omega_1,\sigma)$. Note that $p^*:=p_0+p_1\in
\Lambda_d\setminus[\frac{d-1}{2},\infty)$ and that by convolution
\[
\Gamma(p^*,\omega^*,\sigma):=\Gamma(p_0,\omega_0,\sigma)\star \Gamma(p_1,\omega_1,\sigma)
\]
is a probability measure as well. Since $\omega^*$ is of full rank, we have by Proposition \ref{proplemma}
\ref{proppoint 3b} that $\Gamma(p^*,\omega;\sigma)$ is a probability measure for all
$(\omega,\sigma) \in (\overline S_d^+)^2$. Hence $\Gamma(p^*, \omega ;t\sigma)$ exists for all $(t,\omega,\sigma)\in \mathbb R_+\times
(\overline S_d^+)^2$.

We may now construct a Wishart process by picking some $\alpha\in \overline S_d^+\setminus\{0\}$ and declaring a Markovian transition function by setting for each $(t,x)\in\mathbb R_+\times \overline S_d^+$, $p_t(x,d\xi)$ the probability measure
given by the Laplace transform \begin{equation}\label{eq: FLT Bru1} \int_{\overline S_d^+} e^{-\langle
u,\xi\rangle}p_t(x,d\xi)=\left(\det(I+2t\alpha u) \right)^{-p^*} e^{\tr\left(-u\left(I+2t\alpha
u\right)^{-1}\,x\right)}
\end{equation}
(cf.~\eqref{eq: FLT Bru} for $\beta=0$). Hence $X$ is a Wishart process with constant drift parameter $2p^*$, diffusion coefficient $\alpha$ and zero drift $\beta=0$. But $2p^*\not\geq
(d-1)$, which contradicts Theorem  \ref{enigmatic theorem}. This shows that we indeed must have $\rank(\omega_0)\leq 2p_0+1$.
\end{proof}

\section{Existence of Wishart transition densities}
The aim of this section is to fully characterize the existence of transition densities for Wishart processes.
That is, we investigate whether the transition laws of Wishart processes admit a Lebesgue density.
\begin{theorem}\label{Th: maintheorem}
Let $p>\frac{d-1}{2}$. The following are equivalent
\begin{enumerate}
\item \label{item1} $p_t(x,d\xi)$ has a Lebesgue density $F_{t,x}(\xi)$, for one (hence all) $t>0$.
\item \label{item2} The $d\times d^2$ matrix
\begin{equation}\label{Kalmanmatrix}
[Q^\top\mid \beta Q^\top\mid \dots \mid\beta^{d-1}Q^\top]
\end{equation}
has maximal rank.
\end{enumerate}
Furthermore, if any of the above conditions are satisfied, then $F_{t,x}$ is $C^p(\overline S_d^+)$, for any
$t>0$.
\end{theorem}
\begin{remark}\rm
\begin{itemize}
\item Note that in \eqref{Kalmanmatrix} the matrix $Q$ may be replaced by any matrix $K$ for which $K^\top K=Q^\top Q=\alpha$. This is obvious from the proof of Proposition \ref{Kalman} below.
\item By an inspection of the (Gaussian) transition law of Ornstein-Uhlenbeck processes of the form
\[
Y_t:=e^{\beta t}\left(y+\int_0^te^{-\beta s}Q^\top dW_s\right),
\]
where $W$ is a $d$--dimensional standard Brownian motion, one can infer the well known result
that \ref{item2} characterizes the existence of Lebesgue densities for $Y_t$. In fact, by
 \ref{Kalman}, the covariance matrix of $Y_t$ is non-degenerate, for each $t>0$, so the result holds because $Y$ is a Gaussian process.
\item Condition \ref{item2} is well known in linear control theory, and characterizes the controllability of the linear system
    \[
    \partial_t x(t)=\beta x(t)+Q^\top u(t),\quad x(0)=x_0.
    \]
    That is, let $T>0$. Then for each $x^*\in\mathbb R^d$ there exists a control $u$ such that $x(T)=x^*$. For more details, see \cite[Chapter 3]{trentel}.
\end{itemize}
\end{remark}

The following proposition is a well known ingredient in the characterization of controllability
of linear systems. For the sake of completeness and as service for the reader,
we also prove it here. See, for instance the statements \cite[3.1 to 3.4]{trentel} and their proofs.
\begin{proposition}\label{Kalman}
The following are equivalent:
\begin{enumerate}
\item \label{ch1} For one (hence any) $t>0$, the matrix $\sigma_t^\beta(\alpha)$
is positive definite.
\item \label{ch2} The $d\times d^2$ matrix \eqref{Kalmanmatrix}
has maximal rank.
\end{enumerate}
\end{proposition}
\begin{proof}
By additivity of the integral, it is clear that if $\sigma_t^\beta(\alpha)$
is positive definite for some $t>0$, it is for all $s\geq t>0$.

By Cayley--Hamilton, for each $t\geq 0$ there exist numbers
$a_j(t)$, $j=1,\dots,d-1$ such that
\begin{equation}
e^{tA}=\sum_{j=1}^{d-1} a_j(t)A^j.
\end{equation}
Hence
\begin{equation}\label{eq: chsimpl}
\int_0^t e^{\beta s}Q^\top Q e^{\beta^\top s}ds=\sum_{j,k}g_{jk}(\beta^j Q^\top)(\beta^k Q^\top)
\end{equation}
with
\[
g_{jk}:=\int_0^t a_j(s)a_k(s)ds,
\]
which by construction yields a positive semidefinite matrix $g:=(g_{jk})_{jk}$.

Proof of \ref{ch1}$\Rightarrow$\ref{ch2}: Since $\sigma_t^\beta(Q^\top Q)$ is positive definite, we have
by using eq.~\eqref{eq: chsimpl} that for each $z\in\mathbb R^d\setminus\{0\}$ it holds
\[
\sum_{j,k}g_{jk}(z^\top \beta^j Q^\top)(z^\top\beta^k Q^\top)>0.
\]
But $g$ is positive semidefinite, hence the vector with $(z^\top \beta^j Q^\top)_{j=1}^d$
must be nonzero. Since $z$ was an arbitrary nonzero element of $\mathbb R^d$, we have proved
the rank condition \ref{ch2}.

For the reverse implication, we proceed by an indirect argument. Suppose, there exists $z\neq 0$
such that for all $t>0$, we have that $z^\top\sigma_t^\beta(Q^\top Q)z=0$. Due to the positivity of the integrand
\[
Qe^{\beta^\top t}z=0
\]
for all $t>0$, or equivalently,
\begin{equation}\label{eq: identity}
w^\top Qe^{\beta^\top t}z=0
\end{equation}
for all $w\in\mathbb R^d$, $t>0$. Since the $j$--th derivative of $e^{\beta t}Q^\top$ at $t=0$ equals
$(-1)^j \beta^j Q^\top$, we have by differentiation of
eq.~\eqref{eq: identity} that $w^\top \beta^j Q^\top z=0$ for all $w$ and therefore \ref{ch2} cannot hold.
\end{proof}
\subsection*{Proof of Theorem \ref{Th: maintheorem}}
\begin{proof}
We start with the implication \ref{item2} $\Rightarrow$ \ref{item1}: By Proposition
\ref{Kalman}, we have that for any $t>0$, the matrix $\sigma_t^\beta(\alpha)$
is positive definite. By assumption we have that $p>\frac{d-1}{2}$, and
comparing the Laplace transform \eqref{eq: FLT Bru} of $p_t(x,d\xi)$ with the right side of
\eqref{FLT Mayerhofer Wishart} we realize that $p_t(x,d\xi)\sim\Gamma(p,\omega^\beta_t(x);\sigma_t^\beta(\alpha))$.
Hence by Proposition \ref{lemma: existence wishart dist} \ref{ex wis dist 2} we have that
$p_t(x,d\xi)$ has a Lebesgue density $F_{t,x}(\xi)$, for each $t>0$.

For the converse direction, we proceed by an indirect argument. Assume, for a contradiction,
that the Kalman matrix \eqref{Kalmanmatrix} has rank strictly smaller than $d$. Then by Proposition \ref{Kalman},
$\sigma_t^\beta(\alpha)$ is degenerate for some $t>0$. But then by Proposition \ref{lemma: existence wishart dist}\ref{ex wis dist 3} $p_t(x,d\xi)\sim \Gamma(p,\omega^\beta_t(x);\sigma_t^\beta(\alpha))$ is not absolutely continuous with respect to the Lebesgue density.

So we have shown the equivalence of \ref{item1} and \ref{item2}. The claim concerning the regularity
of the densities $F_{t,x}(\xi)$ is an immediate consequence of the second part of Proposition \ref{lemma: existence wishart dist} \ref{ex wis dist 2}.
\end{proof}

\subsection*{Hitting the boundary revisited}
As application of this section, we prove the following which is stronger to some extent than the assertion of Lemma \ref{lem definite hit}:
\begin{lemma}\label{stronger lem definite hit}
Let $\beta, Q$ be $d\times d$ matrices, and suppose that the Kalman matrix \eqref{Kalmanmatrix} has maximal rank. When $p<\frac{d-1}{2}$, then for any $x\in S_d^+$ we have for the solution of the Wishart SDE that not only $T_x<\infty$ with positive probability but also the stochastic interval $[0, T_x]$ does not contain a deterministic time interval $[0,T]$, $T>0$.
\end{lemma}
\begin{proof}
Assume, for a contradiction, the existence of $x\in S_d^+$ for which $T_x\geq T>0$, where $T$ is a positive quantity. An adaption of Lemma \ref{wish sde is wishart semi} shows that
any solution $X_t$ of the Wishart SDE on $[0,T]$ is a Wishart semimartingale on $[0,T]$. And also from  the proof of Lemma \ref{lem wish semi is wish dist} we see that $X_t\sim \Gamma(p, \omega_t^\beta(x);\sigma_t^\beta(\alpha))$ for all $t\leq T$, where $\alpha=Q^\top Q$. But that means that a Wishart distribution exists with
non-centrality parameter of full rank, and--in view of Proposition \ref{Kalman}--also with scale parameter of full rank, but the shape parameter satisfies $p<\frac{d-1}{2}$. This is impossible in view of
the subsequent statement.
\end{proof}
We cite a special case of Letac and Massam's very recent result \cite{Letac11} on necessary conditions for the parameters of the Wishart distributions. Translated into our notation it reads:
\begin{theorem}\label{th: letac11}
Suppose $\sigma$ and $\omega$ are invertible.  $\Gamma(p,\omega;\sigma)$ can only exist, if $p\geq \frac{d-1}{2}$.
\end{theorem}
\section{Wishart processes on new state spaces}
In a recent work with Cuchiero, Keller-Ressel and Teichmann \cite{CKMT}, the class of affine processes on finite-dimensional symmetric cones
\footnote{These are closed convex selfdual cones on which the linear automorphism group acts transitively} have
been completely characterized. Those also contain affine diffusion processes such as the Wishart processes. Symmetric cones are classified completely \cite{Faraut}, therefore one could try to find SDE realizations as the Wishart SDEs \eqref{wishart full class} on $\overline S_d^+$. However, only in the case of Hermitian matrices the literature provides such realizations. In the latter case we let $W_1, W_2$ be two jointly independent
$d\times n$ Brownian motions $(n\geq d)$, and $y$ be a complex $d\times n$ matrix. Then $X_t:=(y+W_1+iW_2)(\bar y+W_1-iW_2)^\top$ satisfies
\[
dX_t=\sqrt{X_t}dB_t+d\bar B_t^\top \sqrt{X_t}+2p I dt, \quad X_0=y\bar y^\top,
\]
with $B_t$ some $d\times d$ complex Brownian matrix, i.e. $B=B_1+iB_2$, where $B_1, B_2$ are two independent $d\times d$ standard Brownian motions. Here $\bar c$ denotes the complex conjugate of a complex number $c$ and $p=d$. Demni \cite[chapter 2]{Demni2007}  dicusses this case, calling  these processes Laguerre process of integer index. For  general drift parameters $p> d-1$
see the {\bf Laguerre processes} of \cite[chapter 4]{Demni2007}.

\cite{CKMT} delivers for the first time a Wishart process on a non-symmetric cone, namely the dual Vinberg cone. This cone
is given by the five-dimensional subset $K\subset S_3^+$
\[
K=\left\{u=\left(\begin{array}{lll} a & b_1 & b_2\\
b_1& c_1 &0\\
b_2&0&c_2\end{array}\right):\quad u \textrm{  is  positive semi-definite}\right\}
\]
and any element $x\in K$ can be written as
$x=\sum_{i=1}^3 y_iy_i^\top$, where
\[
y_1=\left(\begin{array}{lll} y_{1,1}\\0\\0\end{array}\right),\quad y_2=\left(\begin{array}{lll} y_{2,1}\\y_{2,2}\\0\end{array}\right),\quad y_3= \left(\begin{array}{lll} y_{3,1}\\0\\y_{3,3}\end{array}\right).
\]
We give here a slightly different, yet fully equivalent construction. Let $(B_1,B_2,B_3)$ be a three dimensional standard Brownian motion. We introduce the process $X_t:=\sum_{i=1}^3 Y_{i,t}Y_{i,t}^\top$, where for $i=1,2,3$ we set
\[
Y_{i,t}:=y_i+B_{i,t} e_1
\]
and $e_1=(1,0,0)^\top$ denotes the first canonical basis vector. By Example \ref{wish 1 const} (using $\beta=0,Q=\diag(1,0,0)$ and extending $B_i$ to vector-valued Brownian motions) we know that $X_t$ is a Wishart semimartingale on $S_3^+$, and by construction $X_t$ is supported
on $K$. Hence by the second part of Lemma \ref{wish sde is wishart semi} there exists an enlargement of the original probability space which supports a $3\times 3$ standard Brownian $W$ motion such that
$X_t$ is a weak solution of the Wishart SDE
\[
dX_t=(\sqrt{X_t}dW_t Q+Q^\top dW_t^\top \sqrt X_t)+2p Q^\top Q dt, \quad X_0=x\in K,
\]
where $p=3/2$.

A full understanding of Wishart processes (leave alone general affine processes) on general homogenenous cones is not available at
the date this manuscript is printed.

A final note might be of interest. By using Example \ref{wish 1 const} one obtains Wishart processes with non-convex cone state-space
\[
D_m:=\{u\in \overline S_d^+\mid \rank(u)\leq m\}
\]
and these have drift parameter $p=m/2$, possibly smaller then $(d-1)/2$, thus violating the enigmatic drift condition established in Theorem
\ref{enigmatic theorem} which Wishart processes with support on $\overline S_d^+$ must satisfy. Hence
there is no way to extend the so constructed affine processes on $D_m$
to its convex hull  $\overline S_d^+$.  This also shows that there are more Wishart semimartingales on $\overline S_d^+$
than those which naturally arise from Wishart processes on $\overline S_d^+$. But these are supported
on the strict submanifolds $D_m$ of $\overline S_d^+$, $m<d$.

\section*{Acknowledgement}
I would like to thank Piotr Graczyk and Gerard Letac for invitation to lecture at the CIMPA workshop in Hammamet in September 2011 and for encouraging to write these notes. I am grateful for valuable suggestions for the final version of these notes  from Christa Cuchiero. The author is Marie-Curie Fellow at Deutsche Bundesbank and acknowledges financial support from the European Commission FP7 programme (grant PITN-GA-2009-237984) as well as the
Vienna Science and Technology Fund.

\end{document}